\documentclass[12pt]{amsart}
\usepackage{fullpage}
\usepackage{amssymb,amsmath,amsthm,bbold,mathtools,bm,extpfeil}
\usepackage{enumitem}
\usepackage{comment}
\usepackage{tikz} 
\usetikzlibrary{cd}
\usetikzlibrary{arrows}

\newtheorem{defn}{Definition}[section]
\newtheorem{exmp}[defn]{Example}
\newtheorem{rmk}[defn]{Remark}
\newtheorem{lem}[defn]{Lemma}
\newtheorem{prop}[defn]{Proposition}
\newtheorem{thm}[defn]{Theorem}
\newtheorem{cor}[defn]{Corollary}
\numberwithin{equation}{section}
\newenvironment{sproof}{\begin{proof}[Sketch of proof]}{\end{proof}}
\theoremstyle{plain}

\DeclareMathOperator{\deltaK}{\delta_{\mathrm{K}}}

\newcommand{\bbV}{\mathbb{V}}

\newcommand{\bfJ}{\mathbf{J}}
\newcommand{\bfCEAlgd}{\mathbf{CE}_{\mathbf{Algd}}}
\newcommand{\bfCESp}{\mathbf{CE}_{\mathbf{Sp}}}

\newcommand{\fg}{\mathfrak{g}}

\newcommand{\fJ}{\mathfrak{J}}

\newcommand{\sC}{\mathcal{C}}
\newcommand{\sD}{\mathcal{D}}
\newcommand{\sE}{\mathcal{E}}

\newcommand{\sI}{\mathcal{I}}
\newcommand{\sJ}{\mathcal{J}}
\newcommand{\sK}{\mathcal{K}}
\newcommand{\sL}{\mathcal{L}}
\newcommand{\sM}{\mathcal{M}}
\newcommand{\sN}{\mathcal{N}}
\newcommand{\sO}{\mathcal{O}}
\newcommand{\sT}{\mathcal{T}}

\newcommand{\R}{\mathbb{R}}

\newcommand{\Z}{\mathbb{Z}}
\newcommand{\N}{\mathbb{N}}

\newcommand{\Der}{\mathrm{Der}}

\newcommand{\Sym}{\mathrm{Sym}}

\newcommand{\Hom}{\mathrm{Hom}}
\newcommand{\Gr}{\mathrm{Gr}}
\newcommand{\MC}{\mathrm{MC}}
\newcommand{\CE}{\mathrm{CE}}
\newcommand{\id}{\mathrm{id}}
\newcommand{\End}{\mathrm{End}}
\newcommand{\pbw}{\mathrm{pbw}}
\newcommand{\fem}{\mathrm{fem}}

\newcommand{\wSym}{\widehat{\Sym}}
\newcommand{\wHom}{\widehat{\Hom}}
\newcommand{\wotimes}{\widehat{\otimes}}
\newcommand{\wOmega}{\widehat{\Omega}}

\newcommand{\Linfty}{\texorpdfstring{$L_\infty$}{L-infinity}}
\newcommand{\D}{\texorpdfstring{$D$}{D}}

\usepackage{hyperref}
\usepackage[alphabetic,initials]{amsrefs}

\title{From $L_\infty$ algebroids to $L_\infty$ spaces: Part I}
\author{Alberto S. Cattaneo}
\address{Department of Mathematics, University of Zurich, Winterthurerstrasse 190, CH-8057 Zurich, Switzerland}
\email{cattaneo@math.uzh.ch}
\author{Shuhan Jiang}
\address{Department of Mathematics, University of Zurich, Winterthurerstrasse 190, CH-8057 Zurich, Switzerland}
\email{shuhan.jiang@math.uzh.ch}

\begin{document}
	
	\begin{abstract}
		The notion of $L_\infty$ spaces over dg manifolds is developed. An equivalence between the category of transitive $L_\infty$ algebroids and that of $L_\infty$ spaces is established, and this equivalence detects weak equivalences. Moreover, a faithful functor from $L_\infty$ algebroids to $L_\infty$ spaces is constructed, which also detects weak equivalences.
	\end{abstract}

	\maketitle
	\setcounter{tocdepth}{1}
	\tableofcontents
	
	\section{Introduction}
	
	$L_\infty$ spaces were introduced by Costello as an approach to derived stacks motivated by problems in quantum field theory \cites{costello2011geometric}. In combination with the Batalin--Vilkovisky formalism \cites{Batalin1977,Batalin1981}, $L_\infty$ spaces provide a convenient framework for describing quantization in the formal neighborhood of a family of classical solutions parameterized by a smooth manifold \cites{costello2011renormalization}. Applications to non-linear sigma models recover classical invariants of the target manifold, such as the $\hat{A}$-genus \cites{grady2014one} and the Witten genus \cites{costello2010geometric,costello2011geometric}. In this case, the relevant family of classical solutions consists of constant maps, parameterized by the target manifold.
	
	In many quantum field theory applications, however, the family of classical solutions of interest is not parameterized by a smooth manifold, but rather by a derived smooth manifold. Such situations are not directly covered by Costello’s original formulation: the functor of points of an $L_\infty$ space over a smooth manifold can encode derived structure only in the nilpotent directions \cite{grady2015spaces}. The present paper addresses this limitation by adapting Costello’s notion of $L_\infty$ spaces to the setting of derived manifolds, or more precisely, to that of dg manifolds \cites{Behrend2020thx, carchedi2023derivedmanifoldsdifferentialgraded}.
	
	$L_\infty$ algebroids serve as a homotopical generalization of dg Lie algebroids in which the differential $l_1$, the binary Lie bracket $l_2$, and the unary anchor map $\rho_1$ are extended into hierarchies of higher multilinear operations $\{l_n\}_{n\ge 1}$ and $\{\rho_n\}_{n\ge 1}$ satisfying higher coherent identities, just as $L_\infty$ algebras generalize dg Lie algebras \cites{stasheff1993introduction,lada1995strongly}. They appear in geometry as models of singular foliations and higher Dirac structures \cites{severa2001some, zambon2012algebras, laurent2020universal}, in string theory as geometric formulations of higher gauge field content \cite{sati2012twisted}, and in AKSZ theories as target spaces \cite{alexandrov1997geometry}. $L_\infty$ algebroids also generalize Lie $n$-algebroids \cites{bonavolontaponcin2013,shengzhu2017}. The latter form a special class of $L_\infty$ algebroids whose underlying graded vector bundles are concentrated in degrees $1-n,\dots,0$. 
	More generally, $L_\infty$ algebroids can be defined over a dg manifold $\sM$ by allowing a zeroth anchor map $\rho_0$, which coincides with the cohomological vector field on $\sM$ \cites{bruce2011algebroids, bandiera2020shifted}. 
	
	There exists a close parallel between $L_\infty$ spaces and Lie-theoretic structures in higher differential geometry, as revealed by a result of Grady and Gwilliam \cite{grady2020lie}. They show that to any dg Lie algebroid one can associate an $L_\infty$ space in a functorial manner compatible with weak equivalences.  It is therefore natural to expect that a theory of $L_\infty$ spaces over dg manifolds should develop in correspondence with the existing notion of $L_\infty$ algebroids over dg manifolds. The main results of our paper confirm this expectation.
	
	\subsection{Main results}
	
	To study $L_\infty$ algebroids and $L_\infty$ spaces over dg manifolds in a unified way, we introduce the intermediate notion of \emph{quasi-dg manifolds} from a commutative algebraic perspective. Roughly speaking, a quasi-dg manifold is a ringed space $\sC = (C, \sO_\sC)$ whose structure sheaf $\sO_\sC$ is a sheaf of commutative differential graded algebras, such that, modulo a proper dg ideal $\sI_\sC \subset \sO_\sC$, $\sC$ reduces to an ordinary dg manifold. Using Chevalley–Eilenberg type functors, we establish the following equivalences between categories:
	\begin{enumerate}
		\item $\mathbf{L_\infty Algd} \cong \mathbf{QdgM}^{\mathrm{bp}}$ between the category of $L_\infty$ algebroids over dg manifolds and that of quasi-dg manifolds with base-preserving morphisms;
		\item $\mathbf{L_\infty Algd}_{\mathrm{fib}} \cong \mathbf{QdgM}_{\mathrm{fib}}^{\mathrm{bp}}$ between the category of transitive $L_\infty$ algebroids over dg manifolds and that of fibrant quasi-dg manifolds with base-preserving morphisms;
		\item $\mathbf{L_\infty Sp} \cong \mathbf{QdgM}_{\mathrm{fib}}^{\mathrm{bp}}$ between the category of $L_\infty$ spaces over dg manifolds and that of fibrant quasi-dg manifolds with base-preserving morphisms.
	\end{enumerate}
	Moreover, we prove that all these equivalences detect weak equivalences. Combining the last two, we obtain our first main result:
	\begin{thm}[Corollary \ref{main1}]
		There is an equivalence between the category of transitive $L_\infty$ algebroids over dg manifolds and that of $L_\infty$ spaces over dg manifolds, which detects weak equivalences.
	\end{thm}
	
	Another advantage of working with quasi-dg manifolds is that they naturally admit the language of infinite jets and $D$-modules. In the appendices, we provide a thorough account of these notions in the setting of graded manifolds. With these tools in hand, we introduce the notion of the \emph{jet space} of a quasi-dg manifold, which gives rise to a functor
	\[
	\bfJ^\infty \colon \mathbf{QdgM} \longrightarrow \mathbf{QdgM}_{\mathrm{fib}},
	\]
	called the \emph{jet space functor}. The infinite jet prolongation $j^\infty$ defines a natural transformation from $\bfJ^\infty$ to the identity functor. Since $j^\infty$ is a weak equivalence of quasi-dg manifolds, $\bfJ^\infty$ detects weak equivalences. Together with the preceding categorical equivalences, this yields our second main result:
	
	\begin{thm}[Corollary \ref{main2}]
	    One can construct a faithful functor
		\[
		\mathbf{Fib}\colon \mathbf{L_\infty Algd} \longrightarrow \mathbf{L_\infty Sp} \cong \mathbf{L_\infty Algd}_{\mathrm{fib}},
		\]
		which detects weak equivalences.
	\end{thm}
	
	We refer to $\mathbf{Fib}$ as a \emph{fibrant replacement functor}. The terminology is due to the fact that $L_\infty$ spaces over a fixed dg manifold form a category of fibrant objects.
	We prove this result in a companion paper \cite{jiangCFO}.
	
	\subsection{Future work and applications}
	
	As the title suggests, the present paper is the first in a series. In subsequent work, we will construct a fibrant replacement functor for representations of $L_\infty$ algebroids and apply it to the study of their characteristic classes and shifted symplectic structures. The advantage of passing to $L_\infty$ spaces is that these structures become more transparent and computationally tractable. We will also study the Maurer–Cartan functor of $L_\infty$ spaces over dg manifolds.
	
	One of the main motivations for our work comes from applications to the problem of globalizing the partition function and observables of a quantum field theory over its moduli space of classical solutions within the Batalin–Vilkovisky formalism. This problem has been studied, in a language different from that of $L_\infty$ spaces but closely related to it, in the context of AKSZ sigma models, where one considers the moduli space of constant maps \cites{cattaneofelder2001, bonechi2012poisson, grady2017batalin, cattaneo2019globalization}, as well as in two-dimensional Yang–Mills theory and Chern–Simons theory, where one considers the smooth locus of the moduli space of flat connections \cites{irasomnev2019, mnevwernli2025}. The framework developed in the present paper, together with its sequels, is intended to provide an approach to the globalization problem in the presence of both derived and stacky singularities.
	
	The two main classes of field theories we have in mind for applications are cohomological field theories, as discussed in \cite{jiangjost2026cohft} in a formulation suited to applications of $L_\infty$ spaces over dg manifolds, and AKSZ theories with cotangent derived targets. Both classes admit one-loop exact quantization. The latter is technically more tractable and is expected to produce interesting invariants of dg manifolds \cite{jiangAKSZ}.

	\subsection{Acknowledgments}
	
	The authors would like to thank Ezra Getzler, Owen Gwilliam, Mathieu Stiénon, Maosong Xiang, and Ping Xu for valuable discussions and correspondence.
	
	The authors acknowledge partial support of the SNF Grant No. 200021 227719 and of the Simons Collaboration on Global Categorical Symmetries. This research was (partly) supported by the NCCR SwissMAP, funded by the Swiss National Science Foundation. This article is based upon work from COST Action 21109 CaLISTA, supported by COST (European Cooperation in Science and Technology) (www.cost.eu), MSCA-2021-SE-01-101086123 CaLIGOLA, and MSCA-DN CaLiForNIA-101119552.
	
	\subsection{Notation}
	
	We work throughout in the real smooth setting.
	
	\begin{description}[leftmargin=3cm, labelsep=1cm]
		\item[cdga] Commutative differential graded algebra
		\item[dg manifold] Differential graded manifold
		\item[dg vector bundle] Differential graded vector bundle
		\item[\(\sM\)] Dg manifold (or graded manifold)
		\item[\(\sO_\sM\)] Structure sheaf of \(\sM\)
		\item[\(\sT_\sM\)] Tangent sheaf of \(\sM\)
		\item[\(\wOmega_\sM\)] Sheaf of completed differential forms on \(\sM\)
		\item[\(\sD_\sM\)] Sheaf of differential operators on \(\sM\)
		\item[\(\sJ^\infty_\sM\)] Sheaf of infinite jets on \(\sM\)
	\end{description}

	\section{Dg vector bundles over dg manifolds}
	
	In this section, we review dg manifolds and dg vector bundles over them. In particular, we prove a Whitehead-type theorem for dg vector bundles.
	
	\subsection{Differential graded manifolds}
	
	Let $V=\bigoplus_{i\in \Z} V_i$ be a graded vector space of finite total rank. Suppose that $V$ is concentrated in non-negative degrees, i.e., $V_i = 0$ for all $i < 0$. We denote by $(V_0, \sO_V)$ the ringed space whose structure sheaf assigns each open $U \subset V_0$ the graded commutative algebra
	\[
	\sO_V(U)=C^\infty(U) \otimes_{\Sym(V_0^\vee)} \Sym(V^\vee),
	\]
	where $V^\vee$ denotes the dual graded vector space of $V$, and $\Sym(V^\vee)$ denotes the free graded commutative algebra generated by $V^\vee$.
	
	\begin{defn}
		A \emph{(non-positively) graded manifold} is a ringed space $\sM=(M,\sO_\sM)$, where $M$ is a smooth manifold and $\sO_\sM$ is a sheaf of graded commutative algebras, which is locally isomorphic to $(V_0, \sO_V)$ for some graded vector space $V$ as above.
	\end{defn}
	
	Let $\sM$ be a graded manifold. For each open $U \subset M$, define
	\[
	\sT_\sM(U) \coloneqq \operatorname{Der}(\sO_\sM(U)),
	\]
	where $\operatorname{Der}(\cdot)$ denotes the graded Lie algebra of derivations of a graded commutative algebra. 
	The assignment $\sT_\sM$ defines a sheaf, called the \emph{tangent sheaf} of $\sM$. Global sections of $\sT_\sM$ are called \emph{vector fields} on $\sM$.

	\begin{defn}
		A \emph{dg manifold} is a graded manifold $\sM$ together with a degree $1$ vector field $Q_{\sM}$ over $\sM$ satisfying $[Q_{\sM}, Q_{\sM}] = 0$.
	\end{defn}
	The vector field $Q_{\sM}$ defines a differential on $\sO_{\sM}$, making it into a sheaf of commutative differential graded algebras (cdgas).
	
	\begin{defn}
		A \emph{morphism} of dg manifolds $\sM \to \sN$ is a morphism of ringed spaces
		\[
		(f, f^\sharp) \colon (M, \sO_{\sM}) \longrightarrow (N, \sO_{\sN}),
		\]
		where $f \colon M \to N$ is smooth, $f^\sharp \colon \sO_{\sN} \to f_* \sO_{\sM}$ is a morphism of sheaves of cdgas. Here $f_*$ denotes the direct image functor associated to $f$.
	\end{defn}
	
	By a slight abuse of notation, we denote a dg manifold morphism $(f, f^\sharp)$ simply by $f$.
	
	\begin{rmk}
		By Batchelor's theorem, dg manifolds are essentially curved $L_\infty$ bundles in disguise. They form a category of fibrant objects \cite{Behrend2020thx}: a morphism $f\colon \sM \to \sN$ is called
		\begin{itemize}
			\item  a \emph{weak equivalence} if 
			$f^\sharp_N\colon \sO_\sN(N) \rightarrow \sO_\sM(M)$
			is a quasi-isomorphism;
			\item  a \emph{fibration} if it is a submersion of graded manifolds. 
		\end{itemize}
		Formally inverting weak equivalences in this category produces an $\infty$-category equivalent to the $\infty$-category of derived manifolds \cite{carchedi2023derivedmanifoldsdifferentialgraded}.
	\end{rmk}
	
	\subsection{Differential graded vector bundles}

	Let $\sM$ be a graded manifold. 
	\begin{defn}
		A \emph{graded vector bundle} over $\sM$ is a locally free graded (left) $\sO_\sM$-module of finite total rank.
	\end{defn}
	
	The following result for graded vector bundles over $\sM$ is not trivial, despite first appearances, since $\sM$ is not a locally ringed space. 
	
	\begin{prop}\label{kerrho}
		Let $\sE_1$ and $\sE_2$ be graded vector bundles over $\sM$. Let $\rho\colon \sE_1 \to \sE_2$ be a surjective morphism of graded $\sO_\sM$-modules. Then $\ker \rho$ is a graded vector bundle over $\sM$.
	\end{prop}
	
	We begin by stating a super version of the Quillen–Suslin theorem needed for the proof, which we were unable to locate in the literature.
	
	\begin{lem}[Super Quillen–Suslin theorem]\label{sqs}
		Every finitely generated projective module over the super polynomial ring $\R[x_1, \dots, x_n, \theta_1, \dots, \theta_m]$ is free, where $x_1, \dots, x_n$ denote the even variables and $\theta_1, \dots, \theta_m$ denote the odd variables.
	\end{lem}
	
	\begin{proof}
		For convenience, we denote the exterior algebra $\R[\theta_1,\dots,\theta_m]$ by $\Lambda_m$. The super polynomial ring can then be identified with the tensor product
		\[
		R \coloneqq \R[x_1,\dots,x_n] \otimes_{\R} \Lambda_m.
		\]
		
		Let $I \subset R$ denote the nilpotent ideal $\R[x_1,\dots,x_n] \otimes \Lambda_m^{>0}$, where $\Lambda_m^{>0}$ is the maximal ideal of $\Lambda_m$ generated by the odd variables. Then
		\[
		R/I \cong \R[x_1,\dots,x_n].
		\]
		
		Let $M$ be a finitely generated projective $R$-module. Then $M/IM \cong M \otimes_R R/I$ is a finitely generated projective $R/I$-module. The classical Quillen–Suslin theorem then implies that $M/IM$ is a free $R/I$-module \cites{quillen1976projective,suslin1976projective}.
		
		Since $I$ is nilpotent, $I \subset \operatorname{rad}(R)$. A lemma of Bass (see Lemma 2.4 in \cite{bass1961projective}) states that for projective $R$-modules $P$ and $P'$,
		\[
		P \cong P' \quad \text{if and only if} \quad P/IP \cong P'/IP'.
		\]
		In particular, a projective $R$-module $P$ is free if and only if $P/IP$ is a free $R/I$-module. Applying this criterion to $M$ completes the proof.
	\end{proof}
	
	To apply the super Quillen–Suslin theorem, we use the geometric language of graded vector bundles. By Batchelor's theorem, there exists a positively graded vector bundle $E$ over $M$ such that
	\[
	\sO_\sM \cong \Sym(E^\vee),
	\]
	where $\Sym(E^\vee)$ is understood as the corresponding locally free sheaf on $M$.
	\begin{lem}\label{geomvecbund}
		Let $\sE_1$ and $\sE_2$ be graded vector bundles over $\sM$. Then there exist graded vector bundles $E_1$ and $E_2$ over $M$ such that
		\begin{equation}\label{batchE}
			\sE_i \cong \Sym(E^\vee)\otimes E_i,\quad i=1,2.
		\end{equation}
		Moreover,
		\[
		\Hom(\sE_1,\sE_2)
		\cong
		\Hom
		\bigl(\Sym(E^\vee)\otimes E_1,\,
		\Sym(E^\vee)\otimes E_2\bigr),
		\]
		where the first $\Hom$ denotes the set of morphisms of graded $\sO_\sM$-modules over $\sM$, and the second $\Hom$ denotes the set of morphisms of graded $\Sym(E^\vee)$-modules over $M$.
	\end{lem}
	
	\begin{proof}
		The graded vector bundles $E_i$ are obtained by restricting $\sE_i$ to $M$. The isomorphisms \eqref{batchE} then follow from homotopy invariance of graded vector bundles over split graded manifolds; see Proposition A.6 in \cite{seol2025atiyah}.
		
		We can then identify a graded $\sO_\sM$-module morphism $\rho\colon\sE_1\to\sE_2$ with a family of $C^\infty$-linear sheaf morphisms
		\[
		\rho^n \colon E_1 \longrightarrow \Sym^n(E^\vee)\otimes E_2,\qquad n\ge 0,
		\]
		where all bundles are understood as the corresponding locally free sheaves on $M$. Since each $E_i$ has finite total rank, we can further identify these sheaf morphisms with morphisms of graded vector bundles on $M$ 
		\[
		\underline{\rho}^n \colon E_1 \longrightarrow \Sym^n(E^\vee)\otimes E_2,\qquad n\ge 0.
		\]
		
		Consequently, $\rho$ can be identified with a morphism of graded vector bundles on $M$
		\[
		\underline{\rho}\colon\Sym(E^\vee)\otimes E_1 \longrightarrow \Sym(E^\vee)\otimes E_2,
		\]
		which is $\Sym(E^\vee)$-linear. 
	\end{proof}
	
	We can now prove Proposition \ref{kerrho}.
	
	\begin{proof}[Proof of Proposition \ref{kerrho}]
		It follows from Lemma \ref{geomvecbund} that the short exact sequence of graded $\sO_\sM$-modules 
		\[
		0 \longrightarrow \ker \rho \longrightarrow \sE_1 \longrightarrow \sE_2 \longrightarrow 0.
		\]
		can be identified with the short exact sequence of graded $\Sym(E^\vee)$-modules 
		\[
		0 \longrightarrow \ker \underline{\rho} \longrightarrow \Sym(E^\vee) \otimes E_1 \longrightarrow  \Sym(E^\vee) \otimes E_2 \longrightarrow 0.
		\]
		Thus, $\ker \underline{\rho}$ is a graded vector bundle on $M$ equipped with an induced graded $\Sym(E^\vee)$-module structure from $\Sym(E^\vee) \otimes E_1$. For each $x \in M$, the fiber $(\ker \underline{\rho})_x$ is a direct summand of the free graded $\Sym(E^\vee_x)$-module $\Sym(E^\vee_x) \otimes_{\mathbb{R}} (E_1)_x$ of finite total rank, and hence finitely generated and projective over the super polynomial ring $\Sym(E^\vee_x)$. By Lemma \ref{sqs}, $(\ker \underline{\rho})_x$ is a free graded $\Sym(E^\vee_x)$-module of finite total rank. 
		
		Let $\{e_{x,i}\}_{i=1}^{p}$ be a basis of $(\ker \underline{\rho})_x$ over $\Sym(E_x^\vee)$, where $p$ denotes the total rank of $(\ker \underline{\rho})_x$. Since $\ker \underline{\rho}$ is a graded vector bundle over $M$, we may extend $\{e_{x,i}\}_{i=1}^p$ to sections $\{e_i\}_{i=1}^p$ of $\ker \underline{\rho}|_U$ for some open neighborhood $U \ni x$ satisfying 
		\[
		e_i(x)=e_{x,i}, \qquad i=1, \dots p.
		\]
		
		For convenience, set
		\[
		R\coloneqq\Sym(E^\vee|_U),
		\qquad
		P\coloneqq\ker \underline{\rho}|_U.
		\]
		Then $\{e_i\}_{i=1}^p$ induces a graded $R$-module morphism
		\[
		\phi\colon\bigoplus_{i=1}^p R[-d_i]\longrightarrow P,
		\]
		where $d_i=\deg(e_i)$, $i=1, \dots, p$. After shrinking $U$ further if necessary, choose a local frame $\{e_i\}_{i=p+1}^{p+q}$ of the graded $R$-module $R \otimes E_2|_U$, where $q$ is the total rank of $E_2$. This induces a graded $R$-module isomorphism
		\[
		P\oplus \bigoplus_{i=p+1}^{p+q} R[-d_i]
		\cong
		P \oplus \left(R \otimes E_2|_U\right)
		\cong 
		R \otimes E_1|_U
		\cong
		\bigoplus_{i=1}^{p+q} R[-d_i],
		\]
		where $d_i=\deg(e_i)$, $i=p+1, \dots, p+q$. 
		
		Now define
		\[
		\widetilde{\phi}\colon
		\bigoplus_{i=1}^{p+q} R[-d_i]
		\cong
		\bigoplus_{i=1}^{p} R[-d_i]
		\oplus
		\bigoplus_{i=p+1}^{p+q} R[-d_i]
		\xrightarrow{(\phi,\mathrm{id})}
		P\oplus
		\bigoplus_{i=p+1}^{p+q} R[-d_i]
		\cong
		\bigoplus_{i=1}^{p+q} R[-d_i].
		\]
		By construction, the fiber map $\widetilde{\phi}_x$ is invertible. Hence, after shrinking $U$ once more if necessary, the determinant $\det(\widetilde{\phi}_y) \in R_y$ is invertible for all $y\in U$, and therefore $\widetilde{\phi}$ is an isomorphism over $U$. It follows that $\phi$ is also an isomorphism over $U$.		
	\end{proof}  
	
	\begin{lem}\label{globtens}
		The tensor product $\sE_1 \otimes_{\sO_\sM} \sE_2$ of two graded vector bundles $\sE_1$ and $\sE_2$ is again a graded vector bundle. Moreover,
		\[
		\bigl(\sE_1 \otimes_{\sO_\sM} \sE_2\bigr)(M)
		\cong
		\sE_1(M)\otimes_{\sO_\sM(M)} \sE_2(M).
		\]
	\end{lem}
	
	\begin{proof}
		The first statement is obvious. For the second, recall that on smooth manifolds the global sections functor on vector bundles commutes with tensor products; see Theorem 12.39 of \cite{nestruev2020}. The claim then follows from the identifications in \eqref{batchE}.
	\end{proof}
	
	Let $f\colon \sM_1 \rightarrow \sM_2$ be a graded manifold morphism. Let $\sE_2$ be a graded vector bundle over $\sM_2$. We call the graded $\sO_{\sM_1}$-module
	\[
	f^* \sE_2 \coloneqq \sO_{\sM_1} \otimes_{f^{-1}\sO_{\sM_2}} f^{-1} \sE_2
	\]
	the \emph{pullback} of $\sE_2$ along $f$, where $f^{-1}$ denotes the inverse image functor associated to $f$.
	\begin{lem}\label{globpull}
		$f^* \sE_2$ is a graded vector bundle over $\sM_1$. Moreover,
		\begin{equation}\label{globpullid}
			f^*\sE_2(M_1) \cong \sO_{\sM_1}(M_1) \otimes_{\sO_{\sM_2}(M_2)} \sE_2(M_2).
		\end{equation}
	\end{lem}
	\begin{proof}
		The first statement is obvious. The second holds in the setting of smooth manifolds; see Theorem 12.43 of \cite{nestruev2020}. The result then follows from the identifications in \eqref{batchE}.
	\end{proof}
	\begin{rmk}
		Using \eqref{globpullid}, one obtains a canonical identification
		\begin{equation}\label{pulltan}
			f^*\sT_{\sM_2}(M_1) \cong \Der\left(\sO_{\sM_2}(M_2), \sO_{\sM_1}(M_1)\right),
		\end{equation}
		where $\Der(\sO_{\sM_2}(M_2), \sO_{\sM_1}(M_1))$ denotes the space of vector fields along $f$, i.e., $\R$-linear maps
		\[
		X\colon \sO_{\sM_2}(M_2) \to \sO_{\sM_1}(M_1)
		\]
		satisfying the Leibniz rule
		\[
		X(gh) = X(g)\, f^\sharp_{M_2}(h) + (-1)^{|X||g|} f^\sharp_{M_2}(g)\, X(h),
		\]
		for all $g,h \in \sO_{\sM_2}(M_2)$. See \cite{seol2025atiyah}.
	\end{rmk}
	
	From now on, let $\sM$ be a dg manifold.
	\begin{defn}
		A \emph{dg vector bundle} over $\sM$ is a graded vector bundle $\sE$ equipped with a differential $d_\sE$ making $\sE$ into a dg module over $\sO_\sM$. 
	\end{defn}
	We denote a dg vector bundle by $\sE/\sM$, or simply by $\sE$ when its base is clear from context. 
	
	\begin{exmp}
		The tangent sheaf $\sT_\sM$ is naturally a dg vector bundle, with differential given by the adjoint action of the cohomological vector field:
		\[
		d_{\sT_\sM} \coloneqq [Q_\sM, \cdot].
		\]
	\end{exmp}
	
	Given a dg manifold morphism $f \colon \sM_1 \to \sM_2$ and a dg vector bundle $\sE_2$ over $\sM_2$, the pullback graded vector bundle $f^*\sE_2$ carries a canonical dg vector bundle structure. Its differential is given by
	\[
	d_{f^*\sE_2}(g\otimes e)=
	f^*d_{\sE_2}(g\otimes e)=
	Q_{\sM_1}(g)\otimes e
	+
	(-1)^{|g|}g\otimes f^{-1}d_{\sE_2}(e),
	\]
	for all $g\in \sO_{\sM_1}$ and $e\in f^{-1}\sE_2$.
	
	\begin{defn}
		A \emph{morphism} between dg vector bundles $\sE_1/\sM_1 \rightarrow \sE_2/\sM_2$  consists of a pair $\phi=(f, \phi^\sharp)$, where $f\colon \sM_1 \rightarrow \sM_2$ is a dg manifold morphism and 
		\[
		\phi^\sharp\colon \sE_1 \longrightarrow f^*\sE_2 
		\]
		is a morphism of dg $\sO_{\sM_1}$-modules. Such $\phi$ is called \emph{base-fixing} if $f = \id$.
	\end{defn}

	\begin{exmp}
		Given any dg manifold morphism $f\colon \sM_1 \to \sM_2$, there exists a canonical dg vector bundle morphism $Tf\colon \sT_{\sM_1} \to \sT_{\sM_2}$, with $(Tf)^\sharp\colon  \sT_{\sM_1} \rightarrow  f^*\sT_{\sM_2}$ defined by
		\[
		(Tf)^\sharp(X)(\cdot) = X(f^\sharp(\cdot)),
		\]
		where we use the identification \eqref{pulltan}. 
	\end{exmp}

    \begin{exmp}\label{sdts}
    	We define the following dg vector bundles over $\sM$ associated to a dg vector bundle $\sE/\sM$:
    	
    	\begin{itemize}
    		\item the \emph{suspension}:
    		\[
    		\sE[1] \coloneqq \bigoplus_{i \in \mathbb{Z}} \sE_{i+1},
    		\]
    		equipped with the differential
    		\[
    		d_{\sE[1]} \coloneqq d_\sE;
    		\]
    		
    		\item the \emph{dual}:
    		\[
    		\sE^\vee \coloneqq \mathrm{Hom}_{\sO_\sM}(\sE, \sO_\sM),
    		\]
    		where $\Hom_{\sO_\sM}$ denotes the internal Hom in the category of graded (left) $\sO_\sM$-modules, equipped with the differential
    		\[
    		d_{\sE^\vee}(\alpha)(X)
    		= Q_\sM\bigl(\alpha(X)\bigr)
    		- (-1)^{|\alpha|}\alpha\bigl(d_\sE(X)\bigr);
    		\]
    		
    		\item the \emph{$n$-th tensor power}:
    		\[
    		\sE^{\otimes n}
    		\coloneqq \underbrace{\sE \otimes_{\sO_\sM} \cdots \otimes_{\sO_\sM} \sE}_{n \text{ times}},
    		\]
    		equipped with the differential
    		\begin{align*}
    			d_{\sE^{\otimes n}}(X_1 \otimes \cdots \otimes X_n)
    			&=
    			\sum_{i=1}^n
    			(-1)^{|X_1|+\cdots+|X_{i-1}|}
    			X_1 \otimes \cdots \otimes d_\sE(X_i) \otimes \cdots \otimes X_n;
    		\end{align*}
    		
    		\item the \emph{$n$-th symmetric power}:
    		\[
    		\mathrm{Sym}^n_{\sO_\sM}(\sE)
    		\coloneqq \{X \in \sE^{\otimes n} \mid \sigma(X)=X \ \text{for all } \sigma \in \Sigma_n \},
    		\]
    		where $\Sigma_n$ denotes the symmetric group on $n$ letters, equipped with the differential
    		\[
    		d_{\mathrm{Sym}^n_{\sO_\sM}(\sE)}
    		\coloneqq d_{\sE^{\otimes n}}\big|_{\mathrm{Sym}^n_{\sO_\sM}(\sE)}.
    		\]
    	\end{itemize}
    	
    	All of these constructions are functorial in the category of dg vector bundles over $\sM$.
    \end{exmp}

	\begin{defn}
		A dg vector bundle morphism $\phi\colon \sE_1/\sM_1 \to \sE_2/\sM_2$ is a \emph{weak equivalence} if $f$ is a weak equivalence of dg manifolds and 
		\[
		\phi^\sharp_{M_1}\colon \sE_1(M_1) \rightarrow f^*\sE_2(M_1)
		\]
		is a quasi-isomorphism.
	\end{defn}
	
	A dg vector bundle $\sE/\sM$ is called \emph{weakly contractible} if the canonical morphism from the zero dg vector bundle over $\sM$ to $\sE$ is a weak equivalence; equivalently, the dg $\sO_\sM(M)$-module $\sE(M)$ is acyclic. 
	
	\begin{exmp}
		Let $\phi\colon \sE_1 \to \sE_2$ be a base-fixing morphism of dg vector bundles. The \emph{mapping cone} of $\phi$ is the dg vector bundle
		\[
		C(\phi)
		\coloneqq
		\left(
		\sE_1[1] \oplus \sE_2,
		\quad
		\begin{pmatrix}
			-d_{\sE_1[1]} & 0 \\
			\phi^\sharp & d_{\sE_2}
		\end{pmatrix}
		\right).
		\]
		A standard computation shows that $\phi$ is a weak equivalence if and only if $C(\phi)$ is weakly contractible.
	\end{exmp}

	A dg vector bundle $\sE/\sM$ is called \emph{contractible} if there exists an $\sO_\sM$-linear map $h\colon \sE \to \sE$ of degree $-1$ satisfying
	\[
	d_\sE h + h d_\sE = \id_\sE.
	\]
	We refer to such $h$ as a \emph{(homotopy) contraction}.
	
	\begin{prop}\label{wcc}
		Every weakly contractible dg vector bundle is contractible.
	\end{prop}
	
	To prove this proposition, we need the following useful lemma.
	
	\begin{lem}[Lemma 2.3 in \cite{seol2025atiyah}]\label{hflat}
		Let $\sE/\sM$ be a dg vector bundle. 
		Then $\sE(M)$ is a homotopy flat dg $\sO_\sM(M)$-module; that is, for any acyclic dg $\sO_\sM(M)$-module $A$, the tensor product $A \otimes_{\sO_\sM(M)} \sE(M)$ is again acyclic.
	\end{lem}
	
	\begin{rmk}
		Using the mapping cone argument,  one sees that tensoring with $\sE(M)$ preserves quasi-isomorphisms of dg $\sO_\sM(M)$-modules.
	\end{rmk}
	\begin{sproof}
		The graded commutative algebra $\sO_{\sM}(M)$ is canonically a $C^\infty(M)$-algebra. 
		Recall that there exists a graded vector bundle $E = \bigoplus_{i \in \Z} E_i$ on $M$ such that
		\[
		\sE(M) \cong \sO_{\sM}(M) \otimes_{C^\infty(M)} \Gamma(E).
		\]
		This allows us to define a complete descending filtration on $\sE(M)$ by
		\[
		F^p \sE(M) \coloneqq \bigoplus_{i \ge p} \sO_{\sM}(M) \otimes_{C^\infty(M)} \Gamma(E_i).
		\]
		Since $\sO_{\sM}(M)$ is non-positively graded, the differential $d_\sE$ preserves this filtration, and $\Gr^p d_\sE = Q_\sM \otimes_{C^\infty(M)} 1$ for all $p$. 
		It follows that $\Gr^p\big(A \otimes_{\sO_{\sM}(M)} \sE(M)\big) \cong A \otimes_{C^\infty(M)} \Gamma(E_p)$ is acyclic for all $p$, and hence $A \otimes_{\sO_{\sM}(M)} \sE(M)$ is acyclic.
	\end{sproof}
	
	\begin{proof}[Proof of Proposition \ref{wcc}]
		Let $\sE/\sM$ be a dg vector bundle.  The graded vector bundle
		\[
		\End(\sE) \coloneqq \Hom_{\sO_\sM}(\sE, \sE)
		\]
		carries a canonical dg vector bundle structure, with differential
		\[
		d_{\End(\sE)}(h) = [d_\sE, h], \qquad h \in \End(\sE).
		\]
		One can see that $\sE$ is contractible if and only if $\id_\sE \in \End(\sE)$ is exact. Now assume that $\sE$ is weakly contractible. Note that
		\[
		\End(\sE) \cong \sE \otimes_{\sO_\sM} \sE^\vee.
		\]
		By Lemma \ref{globtens}, there is an isomorphism
		\[
		\End(\sE)(M) \cong \sE(M) \otimes_{\sO_\sM(M)} \sE^\vee(M).
		\]
		The latter complex is acyclic by Lemma \ref{hflat}. It follows that $\id_\sE$ must be exact.
	\end{proof}
	
	We conclude this section with three lemmas that will be used later.
	
	\begin{lem}\label{pbqi}
		Let $f \colon \sM_1 \to \sM_2$ be a weak equivalence of dg manifolds. Let $\sE_i/\sM_i$, $i=1,2$, be dg vector bundles. Under the natural isomorphism
		\[
		\Hom_{\sO_{\sM_1}}(f^*\sE_2,\sE_1)\cong \Hom_{\sO_{\sM_2}}(\sE_2,f_*\sE_1),
		\]
		a morphism $f^*\sE_2 \rightarrow \sE_1$	is a weak equivalence of dg vector bundles if and only if the induced map on global sections
		\[
		\sE_2(M_2) \longrightarrow \sE_1(M_1)
		\]
		is a quasi-isomorphism.
	\end{lem}
	\begin{proof}
		Since $f^\sharp \colon \sO_{\sM_2}(M_2) \to \sO_{\sM_1}(M_1)$ is a quasi-isomorphism, it follows from Lemmas \ref{globpull} and \ref{hflat} that the canonical map
		\[
		\sE_2(M_2) \cong \sE_2(M_2) \otimes_{\sO_{\sM_2}(M_2)} \sO_{\sM_2}(M_2)
		\longrightarrow
		\sE_2(M_2) \otimes_{\sO_{\sM_2}(M_2)} \sO_{\sM_1}(M_1)
		\cong f^*\sE_2(M_1)
		\]
		is a quasi-isomorphism. The claim then follows from the fact that $\sE_2(M_2) \rightarrow \sE_1(M_1)$ is obtained as the composition
		\[
		\sE_2(M_2) \longrightarrow f^*\sE_2(M_1) \longrightarrow \sE_1(M_1),
		\]
		and the 2-out-of-3 property of quasi-isomorphisms.
	\end{proof}

	\begin{lem}\label{sdtspw}
		The operations listed in Example \ref{sdts} preserve base-fixing weak equivalences.
	\end{lem}
	
	\begin{proof}
		For the suspension, the claim is immediate. For the tensor power (and consequently for the symmetric power construction), the result follows directly from Lemma \ref{hflat}. 
		
		More precisely, given a base-fixing weak equivalence $\phi\colon \sE_1 \rightarrow \sE_2$, the induced morphism $\phi \otimes \phi\colon \sE_1 \otimes_{\sO_\sM} \sE_1 \rightarrow \sE_2 \otimes_{\sO_\sM} \sE_2$ can be factored as
		\[
		\sE_1 \otimes_{\sO_\sM} \sE_1 \xlongrightarrow{\id_{\sE_1} \otimes \phi} \sE_1 \otimes_{\sO_\sM} \sE_2 \xrightarrow{\phi \otimes \id_{\sE_2}} \sE_2 \otimes_{\sO_\sM} \sE_2.
		\]
		Since each factor is a weak equivalence, their composition is a weak equivalence. By induction, the induced morphism $\sE_1^{\otimes n} \rightarrow \sE_2^{\otimes n}$ is a weak equivalence for all $n$. Consequently, the claim also holds for the symmetric power, since the symmetrization projector realizes the symmetric power as a direct summand of the tensor power, compatible with the induced morphisms.
		
		We now turn to the dual. Since the dual commutes with the mapping cone (up to suspension), it suffices to show that it preserves weak contractibility. By Proposition \ref{wcc}, every weakly contractible dg vector bundle $\sE$ admits a contraction $h\colon \sE \to \sE$. One readily verifies that its dual map $h^\vee\colon \sE^\vee \to \sE^\vee$ satisfies
		\[
		d_{\sE^\vee} h^\vee + h^\vee d_{\sE^\vee} = \id_{\sE^\vee}.
		\]
		Therefore, $\sE^\vee$ is contractible, hence weakly contractible.
	\end{proof}

	\begin{lem}[Section 3 in \cite{seol2025atiyah}]\label{twe}
		If $f\colon \sM_1 \to \sM_2$ is a weak equivalence of dg manifolds, then $Tf\colon \sT_{\sM_1} \to \sT_{\sM_2}$ is a weak equivalence of dg vector bundles.
	\end{lem}	
	

	\section{\Linfty\ algebroids over dg manifolds}
	
	In this section, we review the notion of $L_\infty$ algebroids over dg manifolds from the Lie algebraic picture. We then introduce a commutative algebraic picture and prove its equivalence to the Lie algebraic one.
	
	\subsection{Lie algebraic perspective}

	Let $\sM$ be a graded manifold and $\sL$ a graded vector bundle over $\sM$. Set
	\[
	\Omega^\bullet_\sL \coloneqq \Sym^\bullet_{\sO_\sM}\!\big(\sL[1]^\vee\big).
	\]
	This graded commutative algebra carries a canonical descending filtration
	\[
	F^p \Omega_\sL \coloneqq \bigoplus_{n \ge p} \Omega^n_\sL, \qquad p \in \Z.
	\]
	Its completion with respect to this filtration is
	\[
	\wOmega_\sL \coloneqq \wSym_{\sO_\sM}\!\big(\sL[1]^\vee\big),
	\]
	where $\wSym_{\sO_\sM}$ denotes the completed symmetric power. When $\sL = \sT_\sM$, we write $\wOmega_{\sT_\sM}$ simply as $\wOmega_\sM$. In this case, $\wOmega_\sM$ is the sheaf of completed differential forms on $\sM$.

	\begin{defn}\label{LAlgd}
		Let $\sM$ be a dg manifold. An \emph{$L_\infty$ algebroid over $\sM$} is a graded vector bundle $\sL$ over $\sM$, together with a filtration preserving differential $D_\sL$ on $\wOmega_\sL$, whose associated graded operator satisfies
		\[
		\Gr D_\sL|_{\sO_\sM} = Q_\sM.
		\]
		The filtered cdga $\CE(\sL)\coloneqq(\wOmega_\sL, D_\sL)$ is called the \emph{Chevalley--Eilenberg algebra} of $\sL$.
	\end{defn}
	
	By a slight abuse of notation, we denote an $L_\infty$ algebroid by $\sL/\sM$, or simply by $\sL$ when the base $\sM$ is clear from context. 
		
	The differential $D_\sL$ determines 
	\begin{itemize}
		\item a sequence of $\sO_\sM$-linear maps of degree $0$
		\[
		\rho_n \colon \Sym^n_{\sO_\sM}(\sL[1]) \longrightarrow \sT_\sM[1], \qquad n \ge 0,
		\]
		called the \emph{multi-anchors} of $\sL$, with $\rho_0 = Q_\sM$;
		
		\item a sequence of $\mathbb{R}$-linear maps of degree $1$
		\[
		l_n \colon \Sym^n_{\sO_\sM}(\sL[1]) \longrightarrow \sL[1], \qquad n \ge 1,
		\]
		called the \emph{multi-brackets} of $\sL$.
	\end{itemize}
	
	More precisely, these maps are determined by the action of $D_\sL$ on the generators of $\wOmega_\sL$:
	\begin{align*}
		&\Omega^0_\sL \longrightarrow \wOmega_\sL,
		\quad
		f \mapsto D_\sL f = \sum_{n \ge 0} \rho_n^\vee(\mathrm{d}f), \\
		&\Omega^1_\sL \longrightarrow \wOmega_\sL,
		\quad
		\alpha \mapsto D_\sL \alpha = \sum_{n \ge 1} l_n^\vee(\alpha), 
	\end{align*}
	where $\rho_n^\vee\colon \Omega^1_\sM \rightarrow \Omega^n_\sL$ is specified by
	\begin{align*} 
		\rho_n^\vee(\alpha)(X_1,\dots,X_n) + (-1)^{|\alpha|} \alpha \left(\rho_n(X_1,\dots,X_n)\right) = 0,
	\end{align*}
	and $l_n^\vee\colon \Omega^1_\sL \rightarrow \Omega^n_\sL$ is specified by
	\begin{align*} 
		&l_n^\vee(\alpha)(X_1,\dots,X_n) + (-1)^{|\alpha|}\alpha \left(l_n(X_1,\dots,X_n)\right) \\
		&= \sum_{i=1}^n (-1)^{|X_i|(|X_1| + \cdots + |X_{i-1}|)} d (\alpha(X_i)) \left(\rho_{n-1}(X_1,\dots,\widehat{X_i},\dots,X_n)\right).
	\end{align*}
	
	They are characterized by the following conditions:
	\begin{itemize}
		\item The failure of $\sO_\sM$-linearity of the $n$-th multi-bracket $l_n$ is measured by the $(n-1)$-th multi-anchor $\rho_{n-1}$:
		\[
		l_n(X_1, \dots, fX_n) = \rho_{n-1}(X_1,\dots, X_{n-1})(f) X_n + (-1)^{|f|(\sum_{i=1}^{n-1} |X_i|+1)} fl_n(X_1, \dots, X_n).
		\]
		For $n=1$, one has 
		\[
		l_1(fX) = Q_\sM(f)X + (-1)^{|f|}f l_1(X).
		\]
		\item The multi-brackets $\{l_n\}_{n \geq 1}$ satisfy the \emph{strong homotopy Jacobi identities}:
		\[
		\sum_{p+q=n} \sum_{\sigma \in \Sigma_n} \frac{\epsilon(\sigma)}{p!q!}	l_{q+1}(l_p(X_{\sigma(1)},\dots),\dots,X_{\sigma(n)}) = 0.
		\]
		For $n=1$, one has
		\[
		l_1 \circ l_1 = 0.
		\] 
		Thus, equipped with $l_1$, $\sL$ becomes a dg vector bundle over $\sM$.
		\item The multi-anchors $\{\rho_n\}_{n \geq 0}$ intertwine the multi-brackets $l_n$ on $\sL[1]$ with the Lie bracket $[\cdot, \cdot]$ on $\sT_\sM[1]$:
		\begin{align*}
			&\sum_{p+q=n} \sum_{\sigma \in \Sigma_n} \frac{\epsilon(\sigma)}{p!q!} 
			\rho_{q+1} \big(
			l_p(X_{\sigma(1)},\dots), \dots,X_{\sigma(n)}
			\big)\\
			&= \frac{1}{2}\sum_{p+q=n} \sum_{\sigma \in \Sigma_n} \frac{\epsilon(\sigma)}{p!q!} 
			\big[
			\rho_{p}(X_{\sigma(1)},\dots),
			\rho_{q}(\dots,X_{\sigma(n)})
			\big],
		\end{align*}
		where $\epsilon(\sigma)$ is the Koszul sign of the permutation $\sigma$.
		For $n=0$, this reduces to $[Q_\sM, Q_\sM]=0$.
		For $n=1$, one has 
		\[
		\rho_1(l_1(X)) = [Q_\sM, \rho_1(X)].
		\]
		In other words, $\rho_1\colon \sL \to \sT_\sM$ is a dg vector bundle morphism.
	\end{itemize}
	
    In fact, the first condition follows from the fact that $D_\sL$ is a derivation, while the second and third follow from the vanishing of $D_\sL^2$ on $\Omega^1_\sL$ and $\Omega^0_\sL$, respectively.
    
    \begin{exmp}
    	A finite-dimensional $L_\infty$ algebra can be seen as an $L_\infty$ algebroid over a point.
    \end{exmp}
    
    \begin{exmp}
    	Let $\sM$ be a dg manifold. There are two canonical $L_\infty$ algebroids over $\sM$:
    	\begin{itemize}
    		\item The zero vector bundle $0_\sM$, equipped with zero multi-brackets and zero multi-anchors, defines an $L_\infty$ algebroid, called the \emph{zero algebroid} on $\sM$. Its Chevalley--Eilenberg algebra is given by $(\sO_\sM, Q_\sM)$.
    		
    		\item Recall that the de Rham differential $\mathrm{d}_{dR}$ is the unique differential on $\wOmega_\sM$ extending
    		\[
    		\Omega^0_\sM \longrightarrow \Omega^1_\sM, \qquad f \mapsto \mathrm{d}f,
    		\]
    		where $\mathrm{d}f(X) = (-1)^{|X||f|} X(f)$ for all vector fields $X$. The Lie derivative $L_X$ is the unique derivation on $\wOmega_\sM$ extending
    		\[
    		\Omega^0_\sM \longrightarrow \Omega^0_\sM, \qquad f \mapsto X(f).
    		\]
    		Now consider the total differential
    		\[
    		D_\sM \coloneqq \mathrm{d}_{dR} + L_{Q_\sM}
    		\]
    		on $\wOmega_\sM$. Obviously, $D_\sM$ preserves the filtration on $\wOmega_\sM$. Moreover,
    		\[
    		\Gr D_\sM|_{\sO_\sM} = Q_\sM.
    		\]
    		Thus, $\sT_\sM$ is an $L_\infty$ algebroid, called the \emph{tangent algebroid} of $\sM$. In terms of multi-anchors and multi-brackets, we have
    		\[
    		\rho_0 = Q_\sM, \quad \rho_1 = \id_{\sT_\sM}, \quad l_1 = [Q_\sM, \cdot], \quad l_2 = [\cdot,\cdot],
    		\]
    		and $\rho_n = l_{n+1}=0$ for all $n \geq 2$.
    	\end{itemize}
    \end{exmp}

	\begin{defn}
		An $L_\infty$ algebroid is called \emph{strict} if $\rho_n = 0$ for all $n \ge 2$. A strict $L_\infty$ algebroid is called an \emph{$L_\infty$ algebra bundle} if, in addition, $\rho_1 = 0$.
	\end{defn}
	
	The compatibility conditions between the multi-brackets $\{l_n\}_{n=1}^\infty$ and the unary anchor $\rho_1$ of a strict $L_\infty$ algebroid $\sL/\sM$ simplify considerably due to the vanishing of the higher multi-anchors. One has
	\begin{align*}
		\rho_1(l_1(X)) = [Q_\sM, \rho_1(X)], \quad
		\rho_{1}\big(l_2(X_{1},X_{2})\big) = \big[\rho_{1}(X_{1}), \rho_{1}(X_{2})\big],
	\end{align*}
	and for all $n \geq 3$,
	\[
	\rho_{1}\big(l_n(X_{1}, \dots, X_{n})\big) = 0.
		\]
	
	\begin{defn}\label{LAlgdM}
		A \emph{(base-preserving) morphism} between two $L_\infty$ algebroids $\sL_1/\sM_1$ and $\sL_2/\sM_2$ consists of a pair $\Phi=(f, \Phi^\sharp)$ where
		\[
		(f, f^\sharp)\colon \sM_1 \rightarrow \sM_2
		\]
		is a dg manifold morphism, and
		\[
		\Phi^\sharp \colon \CE(\sL_2)
		\longrightarrow f_*\CE(\sL_1)
		\]
		is a filtered cdga morphism satisfying
		\[
		\Phi^\sharp|_{\sO_{\sM_2}} = f^\sharp.
		\]
		Such $\Phi$ is called \emph{base-fixing} if $f=\id$.
	\end{defn}

	The filtered cdga morphism $\Phi^\sharp$ is uniquely determined by a sequence $\sO_{\sM_1}$-linear maps of degree $0$
	\[
	\phi^\sharp_n \colon \Sym^n_{\sO_{\sM_1}}(\sL_1[1]) \longrightarrow f^*\sL_2[1], \quad n \ge 1.
	\]
	Here, we use the canonical identifications
	\begin{align*}
		\Hom_{\sO_{\sM_2}}\bigl(\sL_2[1]^\vee, f_* \Sym^n_{\sO_{\sM_1}}(\sL_1[1]^\vee)\bigr)
		&\cong \Hom_{\sO_{\sM_1}}\bigl(f^*\sL_2[1]^\vee, \Sym^n_{\sO_{\sM_1}}(\sL_1[1]^\vee)\bigr) \\
		&\cong \Hom_{\sO_{\sM_1}}\bigl(\Sym^n_{\sO_{\sM_1}}(\sL_1[1]), f^*\sL_2[1]\bigr).
	\end{align*}
	The second isomorphism follows from the fact that pullback commutes with duals for graded vector bundles, together with the canonical identification of a graded vector bundle with its double dual. The latter can be proved using Proposition~\ref{kerrho}.
	
	One can easily check that the condition
	\[
	\Phi^\sharp \circ D_{\sL_2}|_{\Omega^0_{\sL_2}} =  f_*D_{\sL_1} \circ \Phi^\sharp|_{\Omega^0_{\sL_2}}
	\]
	is equivalent to the compatibility of $\{\phi^\sharp_n\}_{n=1}^\infty$ with the multi-anchors on $\sL_1$ and $\sL_2$:
	\[
	Tf^\sharp(\rho_n(X_1,\dots, X_n))=  
	\sum_{j=0}^\infty \sum_{k_1+\cdots+k_j=n} \sum_{\sigma\in \Sigma_n} \frac{\epsilon(\sigma)}{j!k_1!\cdots k_j!}
	f^*\rho_j \big(
	\phi^\sharp_{k_1}(X_{\sigma(1)},\dots),
	\dots,
	\phi^\sharp_{k_j}(\dots,X_{\sigma(n)})
	\big).
	\]
	For $n=0$, this reduces to $Tf^\sharp(Q_{\sM_1})=f^*Q_{\sM_2}$. For $n=1$, one has
	\[
	Tf^\sharp(\rho_1(X))=f^*\rho_1(\phi_1^\sharp(X)).
	\]
	
	The condition
	\[
	\Phi^\sharp \circ D_{\sL_2}|_{\Omega^1_{\sL_2}}
	=
	f_*D_{\sL_1} \circ \Phi^\sharp|_{\Omega^1_{\sL_2}}
	\]
	is more delicate to interpret. If $\sL_1$ and $\sL_2$ are both $L_\infty$ algebra bundles, or if $f=\id$, then it is equivalent to the compatibility of $\{\phi_n^\sharp\}_{n=1}^\infty$ with the multi-brackets on $\sL_1$ and $\sL_2$, namely
	\begin{align*}
		&\sum_{p+q=n}\sum_{\sigma\in\Sigma_n}
		\frac{\epsilon(\sigma)}{p!\,q!}\,
		\phi^\sharp_{q+1}\bigl(
		l_p(X_{\sigma(1)},\dots),\dots,X_{\sigma(n)}
		\bigr) \\
		&=
		\sum_{j=0}^\infty
		\sum_{k_1+\cdots+k_j=n}\sum_{\sigma\in\Sigma_n}
		\frac{\epsilon(\sigma)}{j!\,k_1!\cdots k_j!}\,
		f^*l_j\bigl(
		\phi^\sharp_{k_1}(X_{\sigma(1)},\dots),
		\dots,
		\phi^\sharp_{k_j}(\dots,X_{\sigma(n)})
		\bigr).
	\end{align*}
	For $n=1$, one has
	\[
	\phi_1^\sharp\bigl(l_1(X)\bigr)
	=
	f^*l_1\bigl(\phi_1^\sharp(X)\bigr).
	\]
	In general, however, for $f \neq \id$ and $n>1$, the pullback multi-brackets $f^*l_n$ are not well defined due to the presence of multi-anchors.
	
	\begin{defn}
		An $L_\infty$ algebroid morphism $\Phi\colon \sL_1/\sM_1 \rightarrow \sL_2/\sM_2$ is a \emph{weak equivalence} if the linear term $\phi_1=(f, \phi_1^\sharp)$ is a weak equivalence of the underlying dg vector bundles.
	\end{defn}

	\begin{exmp}
		Let $f\colon\sM_1 \to \sM_2$ be a morphism of dg manifolds. There is a canonical $\sO_{\sM_2}$-linear map
		\[
		f_{dR}^\sharp\colon\wOmega_{\sM_2} \longrightarrow f_*\wOmega_{\sM_1},
		\]
		defined by
		\[
		f_{dR}^\sharp(g)=f^\sharp(g), \qquad
		f_{dR}^\sharp(dg)=f_*d(f^\sharp(g))
		\]
		for $g\in \sO_{\sM_2}$.
		One can easily check that $f_{dR}^\sharp$ defines a filtered cdga morphism from $(\wOmega_{\sM_2}, D_{\sM_2})$ to $f_*(\wOmega_{\sM_1}, D_{\sM_1})$. It follows that $f_{dR} \coloneqq (f, f_{dR}^\sharp)$ defines an $L_\infty$ algebroid morphism between the tangent algebroids.
	\end{exmp}
	
	By Lemma \ref{twe}, $f_{dR}$ is a weak equivalence of tangent algebroids whenever $f$ is a weak equivalence of dg manifolds.
	
	\begin{defn}
		An $L_\infty$ algebroid morphism $\Phi\colon \sL_1/\sM_1 \rightarrow \sL_2/\sM_2$ is called \emph{strict} if $\phi_n^\sharp = 0$ for all $n \geq 2$.
	\end{defn}
	
	The multi-anchors $\{\rho_n\}_{n \geq 1}$ of an $L_\infty$ algebroid $\sL/\sM$ define a base-fixing $L_\infty$ algebroid morphism from $\sL$ to the tangent algebroid $\sT_\sM$, which is strict if and only if $\sL$ is strict. In other words, there is a canonical morphism
	\[
	\CE(\sT_\sM)=(\wOmega_\sM, D_\sM) \longrightarrow \CE(\sL)=(\wOmega_\sL, D_\sL)
	\]
	between filtered cdgas.
	
	\begin{defn}
		An $L_\infty$ algebroid $\sL/\sM$ is called \emph{transitive (or fibrant)} if its unary anchor 
		\[
		\rho_1\colon \sL \longrightarrow \sT_\sM
		\]
		is surjective.
	\end{defn}
	
	\begin{lem}\label{transtr}
		Every transitive $L_\infty$ algebroid is strict up to an isomorphism.
	\end{lem}
	The proof follows by adapting a standard argument from the $L_\infty$ algebra literature.
	\begin{proof}
		Let $\sL/\sM$ be a transitive $L_\infty$ algebroid. We need to construct an $\sO_\sM$-linear automorphism $\Phi^\sharp$ of the filtered graded commutative algebra $\wOmega_\sL$ such that
		\[
		(\Phi^\sharp \circ D_\sL \circ (\Phi^\sharp)^{-1})(f) \in \Omega^0_\sL \oplus \Omega^1_\sL
		\]
		for all $f \in \sO_\sM$.
		
		For $n \ge m \ge 1$, let $\phi_n^m$ denote the composition
		\[
		\Omega^m_\sL \xrightarrow{\Phi^\sharp|_{\Omega^m_\sL}} \wOmega_\sL \longrightarrow \Omega^n_\sL.
		\]
		Recall that $D_\sL(f) = Q_\sM(f)  +  \sum_{m \ge 1} \rho_m^\vee(\mathrm{d}f)$. Then
		\[
		\Phi^\sharp \circ D_\sL(f)
		=
		Q_\sM(f) + \sum_{m \ge 1} \sum_{n \ge m} \phi_n^m\big(\rho_m^\vee(\mathrm{d}f)\big).
		\]
		Note that each $\phi_n^m$ is uniquely determined by $\phi_1^1,\dots,\phi_{n+1-m}^1$. We are therefore led to solve the following equations for $\phi^1_k$'s
		\[
		\sum_{m=1}^n \phi_n^m(\phi_1^1,\dots,\phi_{n+1-m}^1) \circ \rho_m^\vee = 0, \qquad n \ge 2.
		\]
		We impose the initial condition $\phi_1^1 = \mathrm{Id}_{\Omega^1_\sL}$, which implies $\phi_n^n = \mathrm{Id}_{\Omega^n_\sL}$ for all $n \ge 2$. The equations then become
		\[
		\phi_n^1 \circ \rho_1^\vee + \sum_{1<m<n} \phi_n^m \circ \rho_m^\vee = - \rho_n^\vee, \qquad n \ge 2.
		\]
		In particular, for $n=2$, this reduces to
		\[
		\phi_2^1 \circ \rho_1^\vee = -\rho_2^\vee.
		\]
		Since $\sL$ is transitive, the injective map $\rho_1^\vee \colon \Omega^1_\sM \to \Omega^1_\sL$ admits a splitting $h_1 \colon \Omega^1_\sL \to \Omega^1_\sM$. Thus
		\[
		\phi_2^1 = - \rho_2^\vee \circ h_1
		\]
		solves the equation.
	    Proceeding inductively, one constructs $\phi_n^1$ for all $n$, and hence $\Phi^\sharp$.
	\end{proof}
	
	Let $\sL$ be a strict transitive $L_\infty$ algebroid over $\sM$, and let
	\[
	\sL_{red}\coloneqq \ker \rho_1.
	\]
	By Proposition \ref{kerrho}, $\sL_{red}$ is a graded vector bundle over $\sM$.
	\begin{lem}
		For all $n \geq 1$, one has
		\[
		\rho_1(l_n(X_1, \dots, X_n))=0
		\]
		whenever at least one of the elements $X_i \in \sL$ belongs to $\sL_{red}$. 
	\end{lem}
	\begin{proof}
		Recall that for all $n \geq 3$, one has $\rho_1 \circ l_n = 0$. While for $n=1,2$, we have
		\begin{align*}
			\rho_1(l_1(X)) = [Q_\sM, \rho_1(X)], \quad
			\rho_{1}\big(l_2(X_{1},X_{2})\big) = \big[\rho_{1}(X_{1}), \rho_{1}(X_{2})\big].
		\end{align*}
		These expressions vanish whenever $X$ or at least one of $X_1, X_2$ belongs to $\sL_{red}$.
	\end{proof}
	
	In particular, equipped with the induced multi-brackets from $\sL$, $\sL_{red}$ is an $L_\infty$ algebra bundle over $\sM$, called the \emph{isotropy $L_\infty$ algebra bundle}.
	
	\subsection{Commutative algebraic perspective}
	
	Let $V$ be a non-positively graded vector space of finite total rank. Let $(V_0, \sO_V)$ denote the associated ringed space as defined in the previous section. Let $\bbV$ be a graded $\sO_V$-module free of finite total rank. We denote by $(V_0, \sO_\bbV)$ the ringed space whose structure sheaf is given by
	\[
	\sO_\bbV \coloneqq \wSym_{\sO_V}(\bbV^\vee),
	\]
	where $\wSym_{\sO_V}(\bbV^\vee)$ is the completion of $\Sym_{\sO_V}(\bbV^\vee)$ with respect to its canonical descending filtration by symmetric degree.
	\begin{defn}\label{CdgM}
		A \emph{(projected) filtered graded manifold} is a ringed space $\sC =(C, \sO_\sC)$, where $C$ is a smooth manifold and $\sO_\sC$ is a sheaf of graded commutative algebras on $C$, together with an ideal $\sI_\sC$ and a  splitting of the short exact sequence 
		\[
		\begin{tikzcd}
			0 \arrow[r] 
			& \sI_\sC \arrow[r] 
			& \sO_\sC \arrow[r]
			& \sO_\sC/\sI_\sC \arrow[r] 
			& 0
		\end{tikzcd}
		\]
		of sheaves of graded commutative algebras.

		These data are required to satisfy the following condition: for every $x \in C$, there exists an open neighborhood $U \ni x$ and a commutative diagram of ringed spaces
		\[
		\begin{tikzcd}
			(U, \sO_\sC|_U) \arrow[r] \arrow[d] & (V_0, \sO_\bbV) \arrow[d]\\
			(U, (\sO_\sC/\sI_\sC)|_U)  \arrow[r] & (V_0, \sO_V) 
		\end{tikzcd}
		\]
		where the horizontal arrows are isomorphisms.
	\end{defn}
	By definition, $\sO_\sC$ is complete with respect to the $\sI_\sC$-adic filtration
	\[
	F^p \sO_\sC = \sI_\sC^p, \qquad p \in \N.
	\]
	Moreover, the ringed space 
	\[
	\Gr^0 \sC \coloneqq(C, \Gr^0 \sO_\sC = \sO_\sC/\sI_\sC)
	\]
	is a graded manifold, called the \emph{base} of $\sC$.
	
    Let $\sC$ be a filtered graded manifold. For each open $U \subset C$, define
    \[
    \sT_\sC(U) \coloneqq \operatorname{Der}(\sO_\sC(U)).
    \]
    The assignment $\sT_\sC$ defines a sheaf, called the \emph{tangent sheaf} of $\sC$. Global sections of $\sT_\sC$ are called \emph{vector fields} on $\sC$. Note that $\sT_\sC$ carries a natural filtration given by
    \[
    F^p \sT_\sC = \{X \in \sT_\sC\colon X(F^\bullet \sO_\sC) \subset F^{\bullet + p} \sO_\sC \}, \qquad p \in \Z.
    \]
	\begin{defn}
		A \emph{quasi-dg manifold} is a filtered graded manifold $\sC$ equipped with a degree $1$ vector field $Q_\sC$ that satisfies
		\begin{itemize}
			\item $[Q_\sC, Q_\sC] = 0$;
			\item $Q_\sC(\sI_\sC) \subseteq \sI_\sC$, or equivalently,  $Q_\sC \in F^0 \sT_\sC(C)$.
		\end{itemize}
	\end{defn}
	
	Equipped with $Q_\sC$ and the $\sI_\sC$-adic filtration, $\sO_\sC$ becomes a sheaf of filtered cdgas. It follows that $\Gr^0 \sC$, endowed with $\Gr^0 Q_\sC$, is a dg manifold. Note that the morphism of graded commutative algebras $\Gr^0 \sO_\sC \to \sO_\sC$ is not, in general, a morphism of cdgas.
	
	\begin{exmp}
		Let $\sM=(M, \sO_\sM)$ be a dg manifold. 
		\begin{itemize}
			\item $\sM$ equipped with the trivial dg ideal defines a quasi-dg manifold over itself.
			\item The ringed space $\sM_{dR}\coloneqq(M, \wOmega_\sM)$, equipped with the differential $D_\sM = \mathrm{d}_{dR} + L_{Q_\sM}$ and the dg ideal $\wOmega_\sM^{>0}$ of completed differential forms whose projection to functions vanishes, defines a quasi-dg manifold over $\sM$, called the \emph{de Rham space} of $\sM$.
		\end{itemize}
	\end{exmp}

	\begin{defn}
		A \emph{morphism} between quasi-dg manifolds $\sC_1$ and $\sC_2$ consists of a pair 
		$F=(f, F^\sharp)$, where $f\colon C_1 \to C_2$ is a smooth map and $F^\sharp\colon \sO_{\sC_2} \longrightarrow f_*\sO_{\sC_1}$ is a morphism of sheaves of filtered cdgas. 
		
		Such $F$ is called \emph{base-preserving} if the following diagram of sheaves of graded commutative algebras 
		\[
		\begin{tikzcd}
			\sO_{\sC_2} \arrow[r, "F^\sharp"] & f_*\sO_{\sC_1} \\
			\Gr^0 \sO_{\sC_2} \arrow[u] \arrow[r, "\Gr^0 F^\sharp"] & f_*\Gr^0 \sO_{\sC_1} \arrow[u]
		\end{tikzcd}
		\]
		commutes. A base-preserving $F$ is called \emph{base-fixing} if $\Gr^0 F^\sharp = \id$.
	\end{defn}
	We denote by $\Gr^0 F = (f, \Gr^0 F^\sharp)$ the dg manifold morphism between $\Gr^0 \sC_1$ and $\Gr^0 \sC_2$. 
	
	\begin{defn}
		A quasi-dg manifold morphism $F\colon \sC_1 \rightarrow \sC_2$ is called a \emph{weak equivalence} if $F^\sharp_{C_2}\colon \sO_{\sC_2}(C_2) \rightarrow \sO_{\sC_1}(C_1)$ is a filtered quasi-isomorphism; that is, if $\Gr\, F^\sharp_{C_2 }\colon \Gr\, \sO_{\sC_2}(C_2) \rightarrow \Gr\, \sO_{\sC_1}(C_1)$ is a quasi-isomorphism.
	\end{defn}
	
	\begin{lem}
		There exists a unique base-fixing morphism between quasi-dg manifolds
		\[
		H\colon \sC \longrightarrow (\Gr^0 \sC)_{dR}.
		\]
		
		Moreover, for any base-preserving morphism $F\colon \sC_1 \to \sC_2$ between quasi-dg manifolds, the following diagram of quasi-dg manifolds
		\[
		\begin{tikzcd}[column sep=50pt]
			\sC_1 \arrow[r,"F"] \arrow[d, "H"]&\sC_2 \arrow[d, "H"]\\
			(\Gr^0 \sC_1)_{dR} \arrow[r, "dR(\Gr^0 F)"]  &(\Gr^0 \sC_2)_{dR} 
		\end{tikzcd}
		\]
		commutes.
	\end{lem}
	\begin{proof}
		As a filtered graded commutative algebra, $\wOmega_{\Gr^0 \sC}$ is generated by functions and exact $1$-forms.	We define $H^\sharp\colon \wOmega_{\Gr^0 \sC} \to \sO_\sC$ by specifying its action on the generators:
		\[
		H^\sharp(f) = f, \qquad H^\sharp(\mathrm{d}f) = (Q_\sC - \Gr^0 Q_\sC) f, 
		\]
		for all $f \in \Gr^0 \sO_\sC$. It is straightforward to verify that $H$ satisfies the desired properties.
	\end{proof}
	
	\begin{defn}
		A quasi-dg manifold $\sC$ is called \emph{fibrant} if 
		\[
		\Gr\, H^\sharp\colon \Omega_{\Gr^0 \sC} \to \Gr\, \sO_\sC
		\]
	    is injective.
	\end{defn}
	
	Let $\mathbf{L_\infty Algd}$ denote the category of $L_\infty$ algebroids over dg manifolds. Let $\mathbf{QdgM}$ denote the category of quasi-dg manifolds, and $\mathbf{QdgM}^{\mathrm{bp}} \subset \mathbf{QdgM}$ denote the subcategory with the same objects and base-preserving morphisms.
	
	Consider the assignment
	\[
	\bfCEAlgd\colon \mathbf{L_\infty Algd} \longrightarrow \mathbf{QdgM}^{\mathrm{bp}}
	\]
	defined on objects $\sL/\sM$ by
	\[
	\bfCEAlgd(\sL/\sM) = (M, \CE(\sL)),
	\]
	and on morphisms $\Phi=((f, f^\sharp),\Phi^\sharp)\colon \sL_1/\sM_1 \rightarrow \sL_2/\sM_2$ by
	\[
	\bfCEAlgd(\Phi) = (f, \Phi^\sharp)\colon (M_1, \CE(\sL_1)) \longrightarrow (M_2, \CE(\sL_2)).
	\]
	$\bfCEAlgd$ is obviously a well-defined functor, called the \emph{Chevalley--Eilenberg functor}.
	
	\begin{thm}\label{CE1}
		The Chevalley--Eilenberg functor $\bfCEAlgd$ yields an equivalence of categories. Moreover, it detects weak equivalences and fibrant objects.
	\end{thm}	
	\begin{proof}
		By definition, $\bfCEAlgd$ is fully faithful.	We need to show that it is essentially surjective. Let $\sC$ be a quasi-dg manifold. Since we are in the smooth setting, the short exact sequence of locally free graded $\Gr^0 \sO_\sC$-modules
		\[
		0 \longrightarrow F^2\sO_\sC \longrightarrow F^1 \sO_\sC \longrightarrow \Gr^1 \sO_\sC \longrightarrow 0
		\]
		admits a splitting $\Gr^1 \sO_\sC  \rightarrow F^1 \sO_\sC$. Let $\sL = (\Gr^1 \sO_\sC[1])^\vee$. By the universal property of the symmetric algebra, the splitting induces a graded commutative $\Gr^0 \sO_\sC$-algebra morphism
		\[
		\Omega_\sL = \Sym_{\Gr^0 \sO_\sC}(\sL[1]^\vee) \longrightarrow \sO_{\sC},
		\]
		which extends uniquely to a morphism of filtered graded commutative algebras
		\[
		\wOmega_{\sL} \longrightarrow \sO_{\sC}.
		\]
		This morphism is locally an isomorphism and therefore an isomorphism globally. We can use it to transfer the differential $Q_\sC$ on $\sO_\sC$ to a differential $D_\sL$ on $\wOmega_{\sL}$. $\sL$ is then an $L_\infty$ algebroid over $\Gr^0 \sC$. 
		
		Let $\Phi=(f,\{\phi_n^\sharp\}_{n=1}^\infty)\colon \sL_1/\sM_1 \to \sL_2/\sM_2$ be a morphism of $L_\infty$ algebroids. The functor detects weak equivalences: by Lemmas \ref{pbqi} and \ref{sdtspw}, the map
		\[
		\Gr\,  \Phi^\sharp_{M_2}\colon\Sym_{\sO_{\sM_2}(M_2)}\!\big(\sL_2[1]^\vee(M_2)\big)\longrightarrow \Sym_{\sO_{\sM_1}(M_1)}\!\big(\sL_1[1]^\vee(M_1)\big)
		\]
		is a quasi-isomorphism if and only if both
		\[
		f^\sharp_{M_2}\colon\sO_{\sM_2}(M_2)\rightarrow \sO_{\sM_1}(M_1)
		\quad\text{and}\quad
		\phi^\sharp_{M_1}\colon\sL_1[1](M_1)\rightarrow f^*\sL_2[1](M_1)
		\]
		are quasi-isomorphisms.
		
		The functor detects fibrant objects: $\Gr\, H^\sharp\colon\Omega_{\Gr^0 \sC}\to \Gr\,  \sO_\sC$ is injective if and only if 
		\[
		\Gr^1 H^\sharp\colon\Omega^1_{\Gr^0 \sC} \cong (\sT_{\Gr^0 \sC}[1])^\vee\to \Gr^1 \sO_\sC \cong \sL[1]^\vee
		\]
		is injective, which holds precisely when the dual map $\sL \to \sT_{\Gr^0 \sC}$ is surjective.
	\end{proof}

	\section{\Linfty\ spaces over dg manifolds}
	
	In this section, we adapt Costello's notion of $L_\infty$ spaces \cite{costello2011geometric} to the dg manifold setting and prove an equivalence of the categories of transitive $L_\infty$ algebroids and $L_\infty$ spaces over dg manifolds.
	
	\subsection{Curved $L_\infty$ algebras}

	Let $(R,D_R)$ be a cdga and $I\subset R$ a proper dg ideal. Assume that $R$ is complete with respect to the $I$-adic filtration $F^pR=I^p$. Let $\fg$ be a graded $R$-module, endowed with the induced filtration
	\[
	F^p\fg \coloneqq I^p \cdot \fg.
	\]
	Assume further that $\fg$ is complete with respect to this filtration.
	
	\begin{defn}
		A \emph{curved $L_\infty$ algebra} structure on $\fg$ consists of the following data:
		\begin{itemize}
			\item an element of degree $1$
			\[
			l_0 \in F^1 \fg[1];
			\]
			\item a filtration-preserving $\R$-linear map of degree $1$
			\[
			l_1 \colon \fg[1] \to \fg[1];
			\]
			\item for each $n \ge 2$, a filtration-preserving $R$-linear map of degree $1$
			\[
			l_n \colon \Sym_R^n(\fg[1]) \longrightarrow \fg[1].
			\]
		\end{itemize}
		These operations satisfy the following identities:
		\begin{itemize}
			\item (Leibniz rule) for every $r \in R$ and $x \in \fg[1]$,
			\[
			l_1(r x)
			=
			D_R(r)\, x
			+
			(-1)^{|r|} r\, l_1(x);
			\]
			\item (Strong homotopy Jacobi identities) for every $n \ge 0$ and  $x_1,\dots,x_n \in \fg[1]$,
			\[
			\sum_{p+q=n}
			\sum_{\sigma \in \Sigma_n}
			\frac{\epsilon(\sigma)}{p!\, q!}
			\,
			l_{q+1}\big(
			l_p(x_{\sigma(1)},\dots,x_{\sigma(p)}),
			x_{\sigma(p+1)},\dots,x_{\sigma(n)}
			\big)
			=
			0.
			\]
		\end{itemize}
		In particular, one has 
		\[
		l_1(l_0) = 0, \quad \text{and} \quad l_2(l_0, x) + l_1(l_1(x))=0.
		\]
		Accordingly, $l_0$ is called the \emph{curvature} of $\fg$. 
	\end{defn}	
	
	\begin{lem}
		The associated $\N$-graded operator
		\[
		\Gr\, l_1 \colon \Gr\, \fg[1] \longrightarrow \Gr\, \fg[1]
		\]
		is a differential.
	\end{lem}
	
	\begin{proof}
		Since $l_0 \in F^1 \fg[1]$ and $l_n$ preserves the filtration on $\fg[1]$, we have
		\[
		l_1(l_1(x)) = -\, l_2(l_0, x) \in F^{p+1} \fg[1]
		\]
		for any $x \in F^p \fg[1]$. Thus, $\Gr\, l_1 \circ \Gr\, l_1 = 0$.
	\end{proof}
	
	Let $\fg^\vee$ denote the $R$-dual of $\fg$. The filtration on $\fg$ induces a filtration on $\fg^\vee$:
	\[
	F^p \fg^\vee = \{\alpha \in \fg^\vee \mid \alpha(F^q \fg) \subset F^{p+q} R \text{ for all } q \}.
	\]
	\begin{lem}
		If $\fg$ is projective and finitely generated over $R$, then the filtration on $\fg^\vee$ is also induced by the filtration on $R$, i.e.,
		\[
		F^p \fg^\vee = F^p R \cdot  \fg^\vee.
		\]
	\end{lem}
	\begin{proof}
		We only need to show that $F^p \fg^\vee \subset F^p R \cdot \fg^\vee$.
		Without loss of generality, assume that $\fg$ is free of finite rank with basis $\{v_i\}$. If $\alpha \in F^p \fg^\vee$, then $\alpha(v_i) \in F^p R$ for all $i$, so
		\[
		\alpha = \sum_i \alpha(v_i) \cdot v_i^\vee \in F^p R \cdot \fg^\vee,
		\]
		where $\{v_i^\vee\}$ denotes the dual basis.
	\end{proof}
	
	For our purposes, we will henceforth assume that $\fg$ is finitely generated and projective. Consider the following two natural filtrations on  $\Sym_R(\fg[1]^\vee)$:
	\begin{itemize}
		\item Internal filtration defined by
		\[
		F_I^p \Sym_R(\fg[1]^\vee) \coloneqq F^p R \cdot \Sym_R(\fg[1]^\vee) 
		\]
		\item Symmetric filtration defined by
		\[
		F_S^p \Sym_R(\fg[1]^\vee) \coloneqq \bigoplus_{q \geq p} \Sym_R^{q}(\fg[1]^\vee).
		\]
	\end{itemize}
	
	These two filtrations can be combined into a single filtration by convolution:
	\[
	F^p \Sym_R(\fg[1]^\vee) \coloneqq \sum_{i+j = p} F_I^i \Sym_R(\fg[1]^\vee) \odot F_S^j \Sym_R(\fg[1]^\vee),
	\]
	which is the smallest filtration containing both. We use $\wSym_R(\fg^\vee[-1])$ to denote the completion with respect to this filtration. 
	
	The multi-brackets $\{l_n\}_{n=0}^\infty$ on $\fg$ determine a filtration-preserving differential
	\[
	D_\fg \colon \wSym_R(\fg[1]^\vee) \longrightarrow \wSym_R(\fg[1]^\vee).
	\]
	More precisely, $D_\fg$ is determined by its action on the generators
	\[
	D_\fg(r) = D_R(r), \qquad D_\fg(\alpha) = \sum_{n=0}^\infty l_n^\vee(\alpha)
	\]
	for $r \in R$ and $\alpha \in \fg[1]^\vee$. Here $l_n^\vee$ denotes the $R$-linear dual of $l_n$ for $n\neq 1$, while for $n=1$ we define
	\[
	(l_1^\vee(\alpha))(x) = D_R(\alpha(x)) - (-1)^{|\alpha|}\alpha(l_1(x)).
	\]
	The condition $D_\fg^2 = 0$ is equivalent to the strong homotopy Jacobi identities of $\{l_n\}_{n=0}^\infty$.
	\begin{defn}
		We call the filtered cdga
		\[
		\CE(\fg)\coloneqq(\wSym_R(\fg^\vee[-1]), D_\fg)
		\]
		the \emph{Chevalley--Eilenberg algebra} of $\fg$.
	\end{defn}
	
	\begin{defn}
		A \emph{(base-fixing) morphism} between curved $L_\infty$ algebras $\fg_1$ and $\fg_2$ over $(R, D_R)$ is an $R$-linear morphism between filtered cdgas
		\[
		\Phi \colon \CE(\fg_2) \longrightarrow \CE(\fg_1).
		\]
	\end{defn}
	
	Such $\Phi$ is uniquely determined by a degree $0$ element
	\[
	\phi_0 \in F^1 \fg_2[1],
	\]
	and a family of degree $0$ $R$-linear maps
	\[
	\phi_n\colon \Sym_R^n (\fg_1[1]) \rightarrow \fg_2[1], \quad n \geq 1.
	\]
	$\{\phi_n\}_{n=0}^\infty$ intertwines the multi-brackets on $\fg_1$ and $\fg_2$:
	\begin{align*}
		&\sum_{p+q=n} \sum_{\sigma \in \Sigma_n}\frac{\epsilon(\sigma)}{p!q!} \,
		\phi_{q+1}\big(
		l_p(x_{\sigma(1)},\dots),\dots,x_{\sigma(n)}
		\big) \\
		&= \sum_{j=0}^\infty  \sum_{k_1+\cdots+k_j=n} \sum_{\sigma\in \Sigma_n} \frac{\epsilon(\sigma)}{j!k_1!\cdots k_j!}
		\, l_j \big(
		\phi_{k_1}(x_{\sigma(1)},\dots),
		\dots,
		\phi_{k_j}(\dots,x_{\sigma(n)})
		\big).
	\end{align*}
	In particular, for $n=0$ and $n=1$, one has 
	\begin{align*}
		&\phi_1(l_0) = l_0 + \sum_{j=1}^\infty \frac{1}{j!}l_j(\phi_0^{\otimes j}), \\
		&\phi_1(l_1(x)) + \phi_2(l_0,x)
		=
		l_1(\phi_1(x))
		+ \sum_{j=1}^\infty \frac{1}{j!}\, l_{j+1}(\phi_0^{\otimes j}, \phi_1(x)).
	\end{align*}
	\begin{lem}
		The associated $\N$-graded morphism
		\[
		\Gr\, \phi_1\colon\Gr\, \fg_1[1]\longrightarrow \Gr\, \fg_2[1]
		\]
		commutes with the differentials $\Gr\, l_1$.
	\end{lem}
	
	\begin{proof}
		Since $l_0 \in F^1 \fg_1[1]$ and $\phi_0 \in F^1\fg_2[1]$, we have
		\[
		\phi_1(l_1(x)) - l_1(\phi_1(x)) \in F^{p+1}\fg_2[1]
		\]
		for any $x \in F^p \fg_1[1]$. Thus, $\Gr\, \phi_1 \circ \Gr\, l_1 = \Gr\, l_1 \circ \Gr\, \phi_1$.
	\end{proof}
	\begin{defn}
		$\Phi$ is called a \emph{weak equivalence} if 
		\[
		\Gr\, \phi_1 \colon (\Gr\, \fg_1[1], \Gr\, l_1) \longrightarrow(\Gr\, \fg_2[1], \Gr\, l_1)
		\]
		is a quasi-isomorphism.
	\end{defn}
	
	Equivalently, a curved $L_\infty$ algebra morphism $\Phi\colon \fg_1 \rightarrow \fg_2$ can be regarded as a degree $0$ element
	\[
	\alpha \in  F^1\CE(\fg_1)\,\wotimes_R\, \fg_2,
	\]
	where $\wotimes_R$ denotes the completed tensor product of filtered graded $R$-modules. $\alpha$ satisfies the \emph{Maurer–Cartan equation}:
	\[
	\MC(\alpha)\coloneqq 1 \otimes l_0 + (D_{\fg_1} \otimes 1 + 1 \otimes l_1)(\alpha) + \sum_{n\geq 2}^\infty \frac{1}{n!}(1 \otimes l_n)(\alpha^{\otimes n})  = 0.\footnote{The Maurer--Cartan equation is well-defined because for $n \geq 2$,
		\[
		(1 \otimes l_n)(\alpha^{\otimes n}) \in F^n(\CE(\fg_1) \, \wotimes_R \,\fg_2), 
		\]
		and the filtration on $F^1\CE(\fg_1)\, \wotimes_R\, \fg_2$ is complete.}
	\]
	
	The following definition (and Definition \ref{c-pt}) will not be used in the rest of the paper and may be skipped on a first reading.
	\begin{defn}
		Let $(S,D_S)$ be a cdga equipped with a dg ideal $J \subset S$, complete with respect to the $J$-adic filtration, together with a cdga morphism $(R,D_R) \rightarrow (S,D_S)$ sending the dg ideal $I$ of $R$ into $J$.
		
		The set of \emph{$S$-points of $\fg$} is defined as
		\[
		\MC(\fg ; S)\coloneqq \left\{ \alpha \in (J\,\wotimes_R\, \fg)_0\big|\MC(\alpha) = 0 \right\},
		\]
		where $	\MC(\alpha)\coloneqq 1 \otimes l_0 + (D_S \otimes 1 + 1 \otimes l_1)(\alpha) + \sum_{n\geq 2}^\infty \frac{1}{n!}(1 \otimes l_n)(\alpha^{\otimes n})$.
	\end{defn}

	\subsection{$L_\infty$ spaces}
	
	Let $\sM$ be a dg manifold.
	\begin{defn}
		An \emph{$L_\infty$ space} over $\sM$ is a pair $B\fg=(\sM,\fg)$, where $\fg$ is a sheaf of graded $\wOmega_\sM$-modules, locally free of finite total rank, equipped with a curved $L_\infty$ algebra structure over $(\wOmega_\sM,D_\sM)$ with dg ideal $\wOmega_\sM^{>0}$.
	\end{defn}
	
	\begin{exmp}
		An $L_\infty$ space over a point is simply a finite-dimensional $L_\infty$ algebra.
	\end{exmp}
	
	Unlike $L_\infty$ algebroids over dg manifolds, $L_\infty$ spaces over dg manifolds admit pullbacks along morphisms of dg manifolds. Let $f \colon \sM_1 \to \sM_2$ be a morphism of dg manifolds, and let $B\mathfrak{g}_2$ be an $L_\infty$ space over $\sM_2$. We define the pullback $L_\infty$ space $Bf^*\mathfrak{g}_2 = (\sM_1, f^*\fg_2)$ over $\sM_1$ by
	\[
	f^*\fg_2 \coloneqq \wOmega_{\sM_1} \,\wotimes_{f^{-1}\wOmega_{\sM_2}}\, f^{-1}\fg_2.
	\]
	
	\begin{defn}
		A \emph{morphism} $B\fg_1 \to B\fg_2$ of $L_\infty$ spaces is a pair $\Phi=(f, \Phi^\sharp)$, where 
		\begin{itemize}
			\item $f \colon \sM_1 \rightarrow \sM_2$ is a dg manifold morphism;
			\item $\Phi^\sharp\colon \fg_1 \rightarrow f^*\fg_2$ is a morphism of sheaves of curved $L_\infty$ algebras.
		\end{itemize}
		
		Such $\Phi$ is a \emph{weak equivalence} if $f$ is a weak equivalence of dg manifolds and
		\[
		\Phi^\sharp_{M_1}\colon \fg_1(M_1) \rightarrow f^*\fg_2(M_1)
		\]
		is a weak equivalence of curved $L_\infty$ algebras.
	\end{defn}

	A morphism of $L_\infty$ spaces $\Phi\colon B\fg_1 \to B\fg_2$ is equivalent to a degree $0$ element 
	\[
	\alpha \in F^1 \CE(\fg_1)\, \wotimes_{\wOmega_{\sM_1}}\, f^*\fg_2,
	\]
	satisfying the Maurer--Cartan equation.
	\begin{defn}\label{c-pt}
		Let $\sC=(C, \sO_\sC)$ be a quasi-dg manifold. Let $B\fg=(\sM, \fg)$ be an $L_\infty$ space. The set of \emph{$\sC$-points of $B\fg$} is 
		\[
		\MC(B\fg ; \sC)\coloneqq \bigsqcup_{f\colon \Gr^0 \sC \to \sM} \MC(f^*\fg; \sO_\sC),
		\]
		where the disjoint union is taken over all dg manifold morphisms $f\colon \Gr^0 \sC \to \sM$.
	\end{defn}
	
	Let $\mathbf{L_\infty Sp}$ denote the category of $L_\infty$ spaces over dg manifolds. Let $\mathbf{QdgM}^{\mathrm{bp}}_{\mathrm{fib}}$ denote the category of fibrant quasi-dg manifolds with base-preserving morphisms.
	
	Consider the assignment
	\[
	\bfCESp\colon \mathbf{L_\infty Sp} \longrightarrow \mathbf{QdgM}^{\mathrm{bp}}_{\mathrm{fib}}
	\]
	defined on objects $B\fg=(\sM, \fg)$ by
	\[
	\bfCESp(B\fg) = (M, \CE(\fg)), 
	\]
	and on morphisms $\Phi=((f, f^\sharp), \Phi^\sharp)\colon B\fg_1 \rightarrow B\fg_2$ by
	\[
	\bfCESp(\Phi) = (f, \Phi^\sharp)\colon (M_1,  \CE(\fg_1)) \longrightarrow (M_2,  \CE(\fg_2)),
	\]
	where we use the identification
	\[
	\Hom_{\wOmega_{\sM_1}}\bigl(f^*\fg_2[1]^\vee, \CE(\fg_1)\bigr)
	\cong \Hom_{\wOmega_{\sM_2}}\bigl(\fg_2[1]^\vee, f_*\CE(\fg_1)\bigr).
	\]
	$\bfCESp$ is obviously a well-defined functor, called the \emph{Chevalley--Eilenberg functor}.
	
	\begin{thm}\label{CE2}
		$\bfCESp$	yields an equivalence of categories and detects weak equivalences.
	\end{thm}
	\begin{proof}
		By definition, $\bfCESp$ is fully faithful.	We need to show that it is essentially surjective. Let $\sC$ be a fibrant quasi-dg manifold. By Theorem \ref{CE1} and Lemma \ref{transtr}, we know that there exists a strict transitive $L_\infty$ algebroid $\sL$ over $\Gr^0 \sC$ such that
		\[
		(\sO_\sC, Q_\sC) \cong (\CE(\sL), D_\sL)
		\]
		as filtered cdga over $(\wOmega_{\Gr^0 \sC}, D_{\Gr^0 \sC})$. We now set
		\[
		\fg \coloneqq \wOmega_{\Gr^0 \sC} \, \wotimes_{\Gr^0 \sO_\sC}\, \sL_{red},
		\]
		where $\sL_{red} = \ker \rho_1$ and $\rho_1\colon \sL \rightarrow \sT_{\Gr^0 \sC}$ is the unary anchor of $\sL$. Fix a splitting 
		\[
		\sigma\colon \sT_{\Gr^0 \sC}  \longrightarrow \sL
		\]
		of the short exact sequence of graded vector bundles
		\[
		0 \longrightarrow \sL_{red} \longrightarrow \sL \longrightarrow \sT_{\Gr^0 \sC} \longrightarrow 0.
		\]
		We can use $\sigma$ to define a $\Gr^0 \sO_\sC$-linear map
		\[
		\sL[1]^\vee \cong \Omega^1_{\Gr^0 \sC} \oplus \sL_{red}[1]^\vee \hookrightarrow \Sym_{ \wOmega_{\Gr^0 \sC}}(\fg[1]^\vee)
		\]
		By the universal property of the symmetric algebra, this map induces a graded commutative $\Gr^0 \sO_\sC$-algebra morphism
		\[
		\Omega_\sL=\Sym_{\Gr^0 \sO_\sC}(\sL[1]^\vee) \longrightarrow \Sym_{ \wOmega_{\Gr^0 \sC}}(\fg[1]^\vee)
		\]
		which extends uniquely to a morphism of filtered graded commutative algebras
		\[
		\wOmega_\sL=\wSym_{\Gr^0 \sO_\sC}(\sL[1]^\vee) \longrightarrow \wSym_{ \wOmega_{\Gr^0 \sC}}(\fg[1]^\vee),
		\]
	    which is locally an isomorphism and therefore an isomorphism globally. We can use it to transfer the differential $D_{\sL}$ on $\wOmega_{\sL}$ to a differential on $\wSym_{\wOmega_{\Gr^0 \sC}}(\fg[1]^\vee)$. Since all higher anchor maps of $\sL$ vanish, the induced differential $D_{\fg}$ satisfies the desired Leibniz rule
		\[
		D_{\fg}(\omega \alpha) =
		D_{\Gr^0 \sC}(\omega)\,\alpha
		+
		(-1)^{|\omega|}\,\omega\, D_{\fg}(\alpha),
		\]
		where $\omega \in \wOmega_{\Gr^0 \sC}$ and $\alpha \in \wSym_{\wOmega_{\Gr^0 \sC}}(\fg[1]^\vee)$. Thus, $\fg$ is a sheaf of curved $L_\infty$ algebras over $(\wOmega_{\Gr^0 \sC}, D_{\Gr^0 \sC})$, and $B\fg \coloneqq (\Gr^0 \sC, \fg)$ defines an $L_\infty$ space over $\Gr^0 \sC$.
		
		Let $\Phi\colon B\fg_1 \to B\fg_2$ be a morphism of $L_\infty$ spaces, and let $\widetilde{\Phi}\colon \sL_1/\sM_1 \to \sL_2/\sM_2$ be the corresponding morphism of $L_\infty$ algebroids. By Theorem \ref{CE1}, to show that the functor $\bfCESp$ detects weak equivalences, it suffices to show that $\widetilde{\Phi}$ is a weak equivalence of $L_\infty$ algebroids whenever $\Phi$ is a weak equivalence of $L_\infty$ spaces, which follows from Lemma \ref{twe}. 
		
		More precisely, let $\phi_1^\sharp$ denote the linear component of $\Phi^\sharp$. Using the identifications
		\[
		\fg_i \cong \wOmega_{\sM_i} \,\wotimes_{\sO_{\sM_i}}\, (\sL_i)_{red},
		\qquad i=1,2,
		\]
		one sees that
		\[
		\Gr\,\phi_1^\sharp \colon \Gr\,\fg_1[1] \longrightarrow \Gr\,f^*\fg_2[1]
		\]
		is a quasi-isomorphism on global sections if and only if
		\[
		(\sL_1)_{red} \longrightarrow (\sL_2)_{red} 
		\]
		is a weak equivalence of dg vector bundles. Note that $\Phi$ induces a morphism of short exact sequences of dg vector bundles
		\[
		\begin{tikzcd}
			0 \arrow[r] \arrow[d]& (\sL_1)_{red} \arrow[r] \arrow[d] & \sL_1 \arrow[r] \arrow[d] & \sT_{\sM_1} \arrow[r] \arrow[d] & 0 \arrow[d] \\
			0 \arrow[r] & (\sL_2)_{red} \arrow[r] & \sL_2 \arrow[r] & \sT_{\sM_2} \arrow[r] & 0
		\end{tikzcd}.
		\]
		It follows from Lemma \ref{twe} that $(\sL_1)_{red} \rightarrow (\sL_2)_{red}$ is a weak equivalence if and only if $\sL_1 \rightarrow \sL_2$ is a weak equivalence.
	\end{proof}
	\begin{cor}\label{main1}
		There is an equivalence between the category of $L_\infty$ spaces over dg manifolds and that of transitive $L_\infty$ algebroids over dg manifolds, which detects weak equivalences.
	\end{cor}
	
	At first glance, Corollary \ref{main1} may seem surprising, since $L_\infty$ spaces and $L_\infty$ algebroids appear to be of quite different nature. In particular, one may wonder how the curvature of an $L_\infty$ space arises under this equivalence. The following example helps dispel the mystery.
	
	\begin{exmp}
		Let $M$ be a manifold. Let $L$ be a transitive Lie algebroid over $M$; that is, a strict transitive $L_\infty$ algebroid over $M$ whose only non-vanishing multi-bracket is the binary bracket
		\[
		[\cdot, \cdot]_L\colon L \times_M L \rightarrow L.
		\]
		Let $L_{red}$ be the isotropy Lie algebra bundle of $L$, i.e., the kernel of $\rho_L\colon L \rightarrow T_M$. 
		
	    To construct the corresponding $L_\infty$ space over $M$, we follow the proof of Theorem \ref{CE2} to choose a vector bundle splitting
		\[
		\sigma\colon T_M \longrightarrow L
		\]
		of the short exact sequence of Lie algebroids
		\[
		0 \longrightarrow L_{red} \longrightarrow L \xlongrightarrow{\rho_L} T_M \longrightarrow 0.
		\]
		
		This splitting induces a connection on $L_{red}$ by
		\[
		\nabla^\sigma_X(s)\coloneqq [\sigma(X), s]_L.
		\]
		The connection $\nabla^\sigma$ is well-defined since
		\[
		[\nabla^\sigma_X, f](s)= \rho_L \circ \sigma(X)(f) s = X(f)s, \quad  \rho_L(\nabla^\sigma_X(s)) = [X, \rho_L(s)] = 0.
		\]
		Moreover, it is compatible with the bracket 
		\[
		\nabla^\sigma_X([s_1, s_2]_L) = [\nabla^\sigma_X(s_1), s_2]_L + [s_1, \nabla^\sigma_X(s_2)]_L.
		\]
		Its curvature $R_{\nabla^\sigma}$ measures the failure of $\sigma$ to be a Lie algebroid splitting. Explicitly,
		\[
		R_{\nabla^\sigma}(X, Y) = [\sigma(X), \sigma(Y)]_L - \sigma([X,Y]).
		\] 
		
		We are now ready to describe the corresponding $L_\infty$ space $B\fg^\sigma = (M, \fg^\sigma)$ over $M$. For an open $U \subset M$, we set
		\[
		\fg^\sigma(U) \coloneqq \Omega_M(U) \otimes_{C^\infty(U)} \Gamma\bigl(L_{red}|_U\bigr).
		\]		
		The non-vanishing multi-brackets on $\fg^\sigma$ are 
		\[
		l_0 = R_{\nabla^\sigma} \in F^1 \fg^\sigma[1], \qquad l_1 = \mathrm{d}_{\nabla^\sigma}\colon \fg^\sigma[1] \longrightarrow \fg^\sigma[1],
		\]
		and 
		\[
		l_2\colon \Sym^2_{\Omega_M}(\fg^\sigma[1]) \longrightarrow \fg^\sigma[1],
		\]
	    defined by
		\[
		l_2(\alpha \otimes s, \beta \otimes t) = (\alpha \wedge \beta) \otimes [s,t]_L.
		\]
	
	    One can readily verify that the strong homotopy Jacobi identities for $\{l_n\}_{n=0}^2$ hold precisely by virtue of the second Bianchi identity for $R_{\nabla^\sigma}$, the compatibility between $\nabla^\sigma$ and $[\cdot,\cdot]_L$, and the Jacobi identity for $[\cdot,\cdot]_L$.
	\end{exmp}
	
	More generally, let $\sL/\sM$ be a strict transitive $L_\infty$ algebroid, and let $\sigma \colon \sT_{\sM} \to \sL$ be a graded vector bundle splitting of the short exact sequence of strict $L_\infty$ algebroids
	\[
	0 \longrightarrow \sL_{red} \longrightarrow \sL \xlongrightarrow{\rho_1} \sT_{\sM} \longrightarrow 0.
	\]
	Then the curvature of the associated $L_\infty$ space is precisely the obstruction to $\sigma$ being a strict morphism of strict $L_\infty$ algebroids.

	\section{The fibrant replacement functor}
	
	In this section, we construct a faithful functor
	\[
	\mathbf{Fib}\colon \mathbf{L_\infty Algd} \longrightarrow  \mathbf{L_\infty Sp} 
 	\]
	satisfying the following properties:
	\begin{itemize}
		\item it restricts to the identity on $L_\infty$ algebras;
		\item it assigns to each dg manifold its Fedosov resolution;
		\item it detects weak equivalences.
	\end{itemize}
	
	More precisely, $\mathbf{Fib}$ is obtained as the composition
	\[
	\mathbf{L_\infty Algd} \xrightarrow{\bfCEAlgd} \mathbf{QdgM}^{\mathrm{bp}} \xrightarrow{~\mathbf{J}^\infty}~ \mathbf{QdgM}^{\mathrm{bp}}_{\mathrm{fib}} \xrightarrow{\bfCESp^{-1}} \mathbf{L_\infty Sp},
	\]
	where $\mathbf{J}^\infty$ is the main functor we will construct in this section, called the \emph{jet space functor}.
	
	\subsection{Infinite jets and $D$-modules}
	
	Let $\sM$ be a dg manifold.	
	
	\begin{defn}
		A \emph{filtered graded vector bundle} on $\sM$ is a graded (left) $\sO_\sM$-module $\sE$ equipped with a descending filtration
		\[
		\sE \supset \cdots \supset F^{p-1}\sE \supset F^p \sE \supset F^{p+1}\sE \supset \cdots \supset 0,
		\]
		satisfying the following conditions:
		\begin{itemize}
			\item the filtration is exhaustive, separated, and complete, i.e.,
			\[
			\bigcup_{p \in \Z} F^p \sE = \sE,
			\qquad
			\bigcap_{p \in \Z} F^p \sE = 0,
			\qquad
			\sE \cong \varprojlim_p \sE / F^p \sE;
			\]
			\item for each $p \in \Z$, the graded $\sO_\sM$-module
			\[
			\Gr^p \sE = F^p \sE / F^{p+1} \sE
			\]
			is a graded vector bundle, i.e., locally free of total finite rank.
		\end{itemize}
	\end{defn}
	
	We recommend that our readers consult the appendices to familiarize themselves with our notions of differential operators, infinite jet bundles, and $D$-modules on graded manifolds before reading the rest of this section.
	
	\begin{defn}
		A \emph{quasi-dg vector bundle} on $\sM$ is a filtered graded vector bundle $\sE$ equipped with a differential operator
		\[
		D_\sE \colon \sE \to \sE,
		\] 
		of cohomological degree $1$, order $1$, and filtration degree $0$, such that
		\begin{itemize}
			\item $D_\sE$ is a differential, i.e., $D_\sE \circ D_\sE = 0$;
			\item the principal symbol 
			\[
			\sigma_1^0(D_\sE) \in \sT_\sM \otimes_{\sO_\sM} \Gr^0 \End_\sE
			\]
			of $D_\sE$ is given by $ Q_\sM \otimes \id_{\Gr\, \sE}$; or equivalently, 
			\[
			[D_\sE, f]s = Q_\sM(f)s \mod F^{p+1} \sE
			\]
			for all $f \in \sO_\sM$ and $s \in F^p \sE$.
		\end{itemize}
		
		Such $\sE$ is called a \emph{filtered dg vector bundle} if
		\[
		[D_\sE, f] = Q_\sM(f)
		\]
		for all $f \in \sO_\sM$.
	\end{defn}
	
	By a slight abuse of notation, we denote a quasi-dg vector bundle by $\sE/\sM$, or simply by $\sE$ when the base $\sM$ is clear from context. 
	
	\begin{lem}
		Let $\sE$ be a quasi-dg vector bundle. Then
		\[
		\Gr^\bullet D_\sE\colon \Gr^\bullet \sE \rightarrow \Gr^\bullet \sE
		\]
	    makes $\Gr^\bullet \sE$ into a dg vector bundle.
	\end{lem}
	\begin{proof}
		By definition, one has $(\Gr^\bullet D_\sE)^2 = 0$, and
		\[
		[\Gr^\bullet D_\sE, f] = Q_\sM(f)
		\]
		for all $f \in \sO_\sM$.
	\end{proof}
	
	 \begin{exmp}
		A dg vector bundle $\sE$ can be regarded as a filtered dg vector bundle with the trivial filtration
		\[
		F^p \sE =
		\begin{cases}
			\sE, & p \leq 0,\\
			0, & p > 0.
		\end{cases}
		\]
	\end{exmp}
	
	\begin{exmp}\label{FMV1}
		Let $\sC$ be a quasi-dg manifold. The structure sheaf $\sO_\sC$ of $\sC$ defines a quasi-dg vector bundle over its base dg manifold $\Gr^0 \sC$.
	\end{exmp}
    
    \begin{defn}
    	Let $\sE_1$ and $\sE_2$ be quasi-dg vector bundles over a dg manifold $\sM$. A \emph{(base-fixing) morphism}
    	\[
    	\Phi \colon \sE_1 \longrightarrow \sE_2
    	\]
    	is a morphism of graded $\sO_\sM$-modules which preserves the filtrations and the differentials.
    	
    	Such $\Phi$ is called a \emph{weak equivalence} if $\Gr^\bullet \Phi$ is a weak equivalence of dg vector bundles.
    \end{defn}
    
    One can further introduce a category of quasi-dg vector bundles over a quasi-dg manifold. This level of generality becomes necessary when constructing fibrant replacements for representations of $L_\infty$ algebroids. We will not pursue this direction in the present paper.
    
    \begin{exmp}
    	Every base-fixing morphism of dg vector bundles is a morphism of quasi-dg vector bundles (with respect to the trivial filtration).
    \end{exmp}
    \begin{exmp}\label{FMV2}
    	Viewing the structure sheaf of a quasi-dg manifold as a quasi-dg vector bundle over its base dg manifold, every base-fixing morphism $\sC_1 \rightarrow \sC_2$ of quasi-dg manifolds yields a morphism of quasi-dg vector bundles $\sO_{\sC_2} \rightarrow \sO_{\sC_1}$ in the opposite direction.
    \end{exmp}
    Let $\mathbf{QdgM}(\sM)$ and $\mathbf{QdgVB}(\sM)$ denote the categories of quasi-dg manifolds and quasi-dg vector bundles over $\sM$ with base-fixing morphisms, respectively. By Examples \ref{FMV1} and \ref{FMV2}, there is a faithful functor
    \[
    \mathbf{QdgM}(\sM)^{op} \longrightarrow \mathbf{QdgVB}(\sM).
    \]
    \begin{lem}
    	$\mathbf{QdgM}(\sM)^{op} \rightarrow \mathbf{QdgVB}(\sM)$ detects weak equivalences.
    \end{lem}
    \begin{proof}
    	This is tautological from the definitions.
    \end{proof}
    
    Let $\sE$ be a quasi-dg vector bundle over $\sM$. Let $\sJ^\infty_\sE$ be the infinite jet bundle of $\sE$ (viewed as a filtered graded vector bundle). By Proposition \ref{liftE}, the differential $D_{\sE}$ lifts to a differential $j^\infty(D_{\sE})$ on $\sJ^\infty_{\sE}$. Moreover, since $D_{\sE}$ is of order $1$ and filtration degree $0$, $j^\infty(D_{\sE})$ decreases the filtration on $\sJ^\infty_\sE$ by one level.

    Let $\nabla^G$ denote the Grothendieck connection on $\sJ^\infty_\sE$. Since $\nabla^G$ is flat, we have
    \[
    [\nabla^G_{Q_\sM}, \nabla^G_{Q_\sM}] = \nabla^G_{[Q_\sM, Q_\sM]} = 0.
    \] 
    Thus, $\nabla^G_{Q_\sM}$ is a differential on $\sJ^\infty_\sE$, which also lowers the filtration by one level. Since $j^\infty(D_\sE)$ is compatible with the left $\sD_\sM$-action on $\sJ^\infty_\sE$, we have
    \[
    [\nabla^G_{Q_\sM}, j^\infty(D_{\sE})]=0.
    \]
    \begin{prop}\label{jetfdgv}
    	Equipped with the total differential 
    	\[
    	D_{\sJ^\infty_\sE}\coloneqq\nabla^G_{Q_\sM} + j^\infty(D_{\sE}),
    	\]
    	$\sJ^\infty_{\sE}$ becomes a filtered dg vector bundle, called the \emph{infinite jet bundle} of $\sE$.
    \end{prop}
    \begin{proof}
    	$\nabla^G_{Q_\sM}$ is a differential operator of order $1$, while $j^\infty(D_{\sE})$ is of order $0$. Their sum $D_{\sJ^\infty_\sE}$ is therefore a differential operator of order $1$. 
    	
    	We now show that $D_{\sJ^\infty_\sE}$ preserves the filtration on $\sJ^\infty_\sE$. Since both $\nabla^G_{Q_\sM}$ and $j^\infty(D_{\sE})$ lower the filtration by one level, it suffices to show that the sum
    	\[
    	\Gr\, \nabla^G_{Q_\sM}  + \Gr\, j^\infty(D_\sE) \colon  \Gr^\bullet \sJ^\infty_\sE \longrightarrow \Gr^{\bullet -1} \sJ^\infty_\sE
    	\]
    	vanishes. 
    	
    	Recall from Appendix \ref{sec:doj} that 
    	\[
    	\Gr\, \sJ^\infty_\sE \cong \Sym_{\sO_\sM} (\sT_\sM^\vee) \otimes_{\sO_\sM} \Gr\, \sE,
    	\]
    	and that the associated graded operator $\Gr\, \mathrm{d}_{\nabla^G}$ of the exterior covariant derivative $\mathrm{d}_{\nabla^G}$ on
    	\[
    	\Sym_{\sO_\sM} (\sT_\sM[1]^\vee) \otimes_{\sO_\sM} \Sym_{\sO_\sM} (\sT_\sM^\vee) \otimes_{\sO_\sM} \Gr\, \sE,
    	\]
    	is given by minus the universal Koszul differential
    	\[
    	\Gr\, \mathrm{d}_{\nabla^G} = - \delta_K \otimes \id_{\Gr \sE}.
    	\]
    	It follows that
    	\[
    	\Gr\, \nabla^G_{Q_\sM} = - \iota_{Q_\sM} \otimes \id_{\Gr \sE},
    	\]
    	where $\iota_{Q_\sM}\colon \Sym_{\sO_\sM} (\sT_\sM^\vee)  \rightarrow \Sym_{\sO_\sM} (\sT_\sM^\vee)$ denotes the contraction with $Q_\sM$. On the other hand, by Proposition \ref{liftE}, $\Gr\, j^\infty(D_\sE)$ is given by contraction with the principal symbol $\sigma_1^0(D_\sE)$ of $D_\sE$, which precisely cancels $\Gr\, \nabla^G_{Q_\sM}$.
    	
    	To complete the proof, observe that
    	\[
    	[D_{\sJ^\infty_\sE}, f] = [\nabla^G_{Q_\sM}, f] = Q_\sM(f)
    	\]
    	for all $f \in \sO_\sM$.
    \end{proof}
    
    \begin{defn}
    	A \emph{(left) $D$-module} on $\sM$ is a filtered dg vector bundle $\sE$ on $\sM$, together with a $D$-module structure on $\sE$ (viewed as a filtered graded vector bundle) satisfying
    	\[
    	D_\sE(D \cdot s) = L_{Q_\sM}(D) \cdot s + (-1)^{|D|}D \cdot D_\sE(s),
    	\]
    	for all $D \in \sD_\sM$ and $s \in \sE$, where $L_{Q_\sM}(D) = [Q_\sM, D]$ is the Lie derivative of $D$ along $Q_\sM$.
    \end{defn}
    
    \begin{rmk}
    	Equivalently, the compatibility between the left $\sD_\sM$-action and $D_\sE$ reads
    	\[
    	[D_\sE, f] = L_{Q_\sM}(f)
    	\]
    	for $f \in \sO_\sM$, which holds automatically by the definition of a filtered dg vector bundle, together with the condition
    	\[
    	[D_\sE, X \cdot] = [Q_\sM, X]\cdot 
    	\]
    	for $X \in \sT_\sM$, since $\sD_\sM$ is generated by functions and vector fields.
    \end{rmk}
    
    Recall from Appendix~\ref{sec: Dmod} that the de Rham complex of a $D$-module $\sE$ on $\sM$ (viewed as a graded manifold) is given by 
    \[
    dR(\sE)
    =
    \bigl(\wOmega_\sM \,\wotimes_{\sO_\sM}\, \sE,\; \mathrm{d}_{\nabla} \bigr),
    \]
    where $\nabla$ is the flat connection on $\sE$ induced by left $\sD_\sM$-action. Explicitly,
    \[
    \nabla_X(s) = X \cdot s.
    \]
    We now adapt this construction from the graded manifold setting to the dg manifold setting.
    
    \begin{defn}
    	The \emph{de Rham complex} of a $D$-module $\sE$ on $\sM$ is the following quasi-dg vector bundle on $\sM$:
    	\[
    	dR(\sE)\coloneqq \bigl(\wOmega_\sM \,\wotimes_{\sO_\sM}\, \sE,\; \mathrm{d}_{\nabla} + L_{D_\sE} \bigr),
    	\]
    	where $\nabla$ is the flat connection on $\sE$ induced by left $\sD_\sM$-action, and $L_{D_\sE}$ is given by
    	\[
    	L_{D_\sE}(\alpha \otimes s) = L_{Q_\sM}(\alpha) \otimes s + (-1)^{|\alpha|} \alpha \otimes D_\sE(s).
    	\]
    \end{defn}
    
    \begin{prop}
    	The de Rham complex $dR(\sE)$ is well-defined.
    \end{prop}
    \begin{proof}
    	By definition, $L_{D_\sE}$ is well-defined and squares to zero. Moreover,
    	\[
    	[L_{D_\sE}, f] s = Q_\sM(f)s, \quad \mathrm{d}_{\nabla}(s) = 0 \mod F^{p+1} dR(\sE),
    	\]
    	for all $s \in F^p dR(\sE)$. It remains to verify that $[\mathrm{d}_\nabla, L_{D_\sE}] = 0$. Since $[\mathrm{d}_\nabla, L_{Q_\sM}^{\nabla}] = 0$, it suffices to show that
    	\[
    	\nabla(D_\sE(s))+ L_{D_\sE} \nabla(s) = 0.
    	\]
    	Let $x^\mu$ be local coordinates on $\sM$. We compute
    	\begin{align*}
    		dx^\mu  \otimes \nabla_{\frac{\partial}{\partial x^\mu}}(D_\sE(s)) + L_{D_\sE} (dx^\mu \otimes  \nabla_{\frac{\partial}{\partial x^\mu}}(s)) &= L_{Q_\sM}(dx^\mu) \otimes \nabla_{\frac{\partial}{\partial x^\mu}}(s)  - dx^\mu \otimes [D_\sE, \nabla_{\frac{\partial}{\partial x^\mu}}](s) \\
    		&= L_{Q_\sM}(dx^\mu) \otimes \nabla_{\frac{\partial}{\partial x^\mu}}(s) - dx^\mu \otimes \nabla_{[Q_\sM, \frac{\partial}{\partial x^\mu}]}(s) \\
    		&=0,
    	\end{align*}
    	which completes the proof.
    \end{proof}
    
    One readily sees that the infinite jet bundle $\sJ^\infty_\sE$ of a quasi-dg vector bundle $\sE$ carries a compatible $D$-module structure. More precisely, we have
    \[
    [D_{\sJ^\infty_\sE}, X \cdot] =  [D_{\sJ^\infty_\sE}, \nabla^G_X] = [\nabla^G_{Q_\sM}, \nabla^G_X] = \nabla^G_{[Q_\sM, X]} = [Q_\sM, X] \cdot,
    \]
    where we use Proposition \ref{liftE} and the flatness of $\nabla^G$. For simplicity, we denote by 
    \[
    (\fJ^\infty_{\sE} \coloneqq \wOmega_\sM \,\wotimes_{\sO_\sM}\, \sJ^\infty_{\sE},\; D_{\fJ^\infty_\sE} \coloneqq \mathrm{d}_{\nabla^G} + L_{D_{\sJ^\infty_\sE}} )
    \] 
    the de Rham complex of $\sJ^\infty_\sE$. Note that the differential $L_{D_{\sJ^\infty_\sE}}$ can be written as
    \[
    L_{D_{\sJ^\infty_\sE}} = L_{Q_\sM}^{\nabla^G} + \id_{\wOmega_\sM}  \otimes j^\infty(D_\sE).
    \]
       
    \begin{prop}\label{drjetfdgv}
    	The associated graded operator $\Gr\, D_{\fJ^\infty_\sE}$ on
    	\[
    	\Gr\, \fJ^\infty_\sE \cong \sK_\sM \otimes_{\sO_\sM} \Gr\, \sE, \quad \text{with} \quad \sK_\sM\coloneqq \Omega_\sM \otimes_{\sO_\sM} \Sym_{\sO_\sM} (\sT_\sM^\vee),
    	\]
    	is given by
    	\[
    	\Gr\, D_{\fJ^\infty_\sE} = - \delta_K \otimes \id_{\Gr\, \sE} + L_{Q_\sM} \otimes \id_{\Gr\, \sE} + \id_{\sK_\sM} \otimes \Gr\, D_\sE,
    	\]
    	where $L_{Q_\sM}$ denotes the Lie derivative on $\sK_\sM$ along $Q_\sM$.
    \end{prop}

    \begin{proof}
    	It suffices to show that 
    	\[
    	\Gr\, (L_{Q_\sM}^{\nabla^G} + \id_{\wOmega_\sM}  \otimes j^\infty(D_\sE)) =  L_{Q_\sM} \otimes \id_{\Gr\, \sE} + \id_{\sK_\sM} \otimes \Gr\, D_\sE.
    	\]
    	
    	The most straightforward way to verify this is by working in local coordinates. Let $x^\mu$ be local coordinates on $\sM$, $dx^\mu$ and $y^\mu$ be the corresponding fiber coordinates on $\sT_\sM[1]^\vee$ and     $\sT_\sM^\vee$, and $u^a$ be local fiber coordinates on $\sE$. Locally identifying $\sJ^\infty_\sE$ with $\wSym_{\sO_\sM}(\sT_\sM^\vee) \, \wotimes_{\sO_\sM}\, \sE$ using the trivial connection, we can write
    	\[
    	\mathrm{d}_{\nabla^G} = dx^\mu \frac{\partial}{\partial x^\mu} - \left(dx^\mu \frac{\partial}{\partial y^\mu}\right) \otimes \id_\sE.
    	\]
    	Thus, $L^{\nabla^G}_{Q_\sM}$ can be expressed as
    	\[
    	L^{\nabla^G}_{Q_\sM} = Q_\sM^\mu \frac{\partial}{\partial x^\mu} - \left(dx^\nu \frac{\partial Q_\sM^\mu}{\partial x^\nu} \frac{\partial}{\partial dx^\mu} + Q_\sM^\mu \frac{\partial}{\partial y^\mu}\right) \otimes \id_\sE.
    	\]
    	
    	On the other hand, $D_\sE$ can be locally written as
    	\[
    	D_\sE = D_\sE^\mu(x,u) \frac{\partial}{\partial x^\mu} + D_\sE^0(x, u),
    	\]
    	with $D_\sE^0$ denoting the zeroth-order part of $D_\sE$. By definition, we have
    	\[
    	D_\sE^\mu(x,u) = Q_\sM^\mu(x) + \cdots,
    	\]
    	where $\cdots$ denotes extra terms of filtration degrees $\geq 1$. Then we can write
    	\[
    	\Gr\, D_\sE = Q_\sM^\mu \frac{\partial}{\partial x^\mu} + \Gr\, D_\sE^0,
    	\qquad
    	j^\infty(D_\sE) = \sum_{\alpha} \frac{y^\alpha}{\alpha!} \left( \frac{\partial  D_\sE^\mu}{\partial x^\alpha} \frac{\partial}{\partial y^\mu} + \frac{\partial  D_\sE^0}{\partial x^\alpha} \right).\footnote{See Appendix \ref{sec:doj} for the multi-index notation.}
    	\]
    	Since $y^\mu$ has filtration degree $1$, we have
    	\[
    	\id_{\wOmega_\sM}  \otimes j^\infty(D_\sE) =  \id_{\wOmega_\sM}  \otimes \left(Q_\sM^\mu \frac{\partial}{\partial y^\mu} + y^\nu \frac{\partial Q_\sM^\mu}{\partial x^\nu} \frac{\partial}{\partial y^\mu} \right)  \otimes \id_\sE + \id_{\wOmega_\sM} \otimes  \Gr\, D^0_\sE  + \cdots,
    	\]
    	where $\cdots$ denotes extra terms of filtration degrees $\geq 1$. 
    	
    	To sum up, we have
    	\begin{align*}
    		\Gr\, (L_{Q_\sM}^{\nabla^G} + 1 \otimes j^\infty(D_\sE))  &= \left(Q_\sM^\mu \frac{\partial}{\partial x^\mu} - dx^\nu \frac{\partial Q_\sM^\mu}{\partial x^\nu} \frac{\partial}{\partial dx^\mu} + y^\nu \frac{\partial Q_\sM^\mu}{\partial x^\nu} \frac{\partial}{\partial y^\mu} \right) \otimes \id_{\Gr\, \sE}  \\
    		& + \id_{\sK_\sM} \otimes \left(Q_\sM^\mu \frac{\partial}{\partial x^\mu} +  \Gr\, D_\sE^0 \right) \\
    		&=  L_{Q_\sM} \otimes \id_{\Gr\, \sE} + \id_{\sK_\sM} \otimes \Gr\, D_\sE,
    	\end{align*}
    	which completes the proof.
    \end{proof}
    
    Let $j^\infty\colon \sE \to \sJ^\infty_\sE$ denote the infinite jet prolongation of $\sE$. By abuse of notation, we also denote by
    \[
    j^\infty\colon \sE \longrightarrow \fJ^\infty_\sE
    \]
    the composite $\sE \xrightarrow{j^\infty} \sJ^\infty_\sE \hookrightarrow \fJ^\infty_\sE$.
    
    \begin{prop}\label{fqifjse}
    	The $\R$-linear map
    	\[
    	j^\infty_M\colon \sE(M) \longrightarrow \fJ^\infty_\sE(M)
    	\]
    	is a filtered quasi-isomorphism; that is, 
    	\[
    	\Gr\, j^\infty_M\colon \Gr\, \sE(M) \longrightarrow \Gr\, \fJ^\infty_\sE(M) \cong \sK_\sM(M) \otimes_{\sO_\sM(M)} \Gr\, \sE(M)
    	\]
    	is a quasi-isomorphism.
    \end{prop}
    \begin{proof}
    	$\Gr\, j^\infty$ is given by the canonical inclusion of $\Gr\, \sE$. It suffices to prove that the complex
    	\[
    	(\Gr\, \fJ^\infty_\sE(M) \cong \sK_\sM(M) \otimes_{\sO_\sM(M)} \Gr\, \sE(M), \Gr\, D_{\fJ^\infty_\sE})
    	\]
    	is quasi-isomorphic to $(\Gr\, \sE(M), \Gr\, D_\sE)$. The above complex can be regarded as a double complex with horizontal and vertical differentials 
    	\[
    	d_h\coloneqq \delta_K \otimes \id_{\Gr\, \sE}, \qquad d_v \coloneqq L_{Q_\sM} \otimes \id_{\Gr\, \sE} + \id_{\sK_\sM} \otimes \Gr\, D_\sE.
    	\]
    	Since the cohomology of $\delta_K$ is $\sO_\sM(M)$ and $\Gr^\bullet \sE(M)$ is finitely generated and projective over $\sO_\sM(M)$, we have
    	\[
    	H_{d_h}(\Gr\, \fJ^\infty_\sE(M)) \cong \Gr\, \sE(M).
    	\]
    	Now observe that
    	\[
    	d_v|_{\Gr\, \sE(M)} = \Gr\, D_\sE.
    	\]
    	Thus, the $E_1$ page of the spectral sequence of this double complex is isomorphic to
    	\[
    	(\Gr\, \sE(M), \Gr\, D_\sE),
    	\]
    	which completes the proof.
    \end{proof}
    
	\subsection{The jet space functor}
	
	Let $\sM$ be a dg manifold. Let $\sJ^\infty_\sM$ be the sheaf of infinite jets on $\sM$, i.e., the infinite jet bundle of the structure sheaf $\sO_\sM$.  Applying the construction in the previous subsection to $\sO_\sM$, regarded as a quasi-dg vector bundle, we obtain the quasi-dg vector bundle
	\[
	\fJ^\infty_\sM\coloneqq \wOmega_\sM \,\wotimes_{\sO_\sM}\, \sJ^\infty_\sM,
	\]
	equipped with the total differential
	\[
	D_{\fJ^\infty_\sM} = \mathrm{d}_{\nabla^G} + L_{Q_\sM},
	\]
	where $L_{Q_\sM}$ denotes the (total) Lie derivative on $\fJ^\infty_\sM$ along $Q_\sM$
	\begin{equation} \label{total_Lie}
		L_{Q_\sM} = L^h_{Q_\sM} + L^v_{Q_\sM},
	\end{equation}
	with the horizontal and vertical Lie derivatives given by the formulas
	\[
	L^h_{Q_\sM} = L_{Q_\sM}^{\nabla^G}, \qquad L^v_{Q_\sM} = \id_{\wOmega_\sM} \otimes j^\infty(Q_\sM),
	\]
	respectively.
	
	Recall from Appendix \ref{sec:doj} that $\sJ^\infty_\sM$ is a filtered graded commutative algebra. It follows that $\fJ^\infty_\sM$ is also a filtered graded commutative algebra. It is straightforward to verify that $D_{\fJ^\infty_\sM}$ is a derivation of $\fJ^\infty_\sM$. Therefore, $\fJ^\infty_\sM$ is a filtered cdga.

	\begin{prop}
		The pair $\sM_{jet}\coloneqq(M, \fJ^\infty_\sM)$ defines a fibrant quasi-dg manifold over $\sM$, called the \emph{jet space} of $\sM$. Moreover, the morphism
		\[
		\sM_{jet}=(M, \fJ^\infty_\sM) \xlongrightarrow{(\id_M, j^\infty)} \sM=(M, \sO_\sM)
		\]
		is a weak equivalence in $\mathbf{QdgM}$.\footnote{Note that $(\id_M, j^\infty)$ is not base-preserving.}
	\end{prop}
	
	\begin{proof}
		We need to show that $\sM_{jet}$ is well-defined. We have
		\[
		\Gr^0 (\fJ^\infty_\sM) \cong \Omega^0_\sM \otimes_{\sO_\sM} \Gr^0(\sJ^\infty_\sM) \cong \sO_\sM.
		\]
		Both the splitting $\sO_\sM \rightarrow \fJ^\infty_\sM$ and the morphism $H^\sharp\colon \wOmega_\sM \rightarrow \fJ^\infty_\sM$ are given by the canonical inclusions. $\Gr\, H^\sharp$ is obviously injective.
	
		The morphism $(\id_M, j^\infty)$ is well-defined by Lemma~\ref{prol}. It is a weak equivalence by Proposition~\ref{fqifjse}.
	\end{proof}
	
	\begin{rmk}
		Using a formal exponential map 
		\[
		\mathrm{fem}\colon \sJ^\infty_\sM \longrightarrow \wSym_{\sO_\sM}(\sT_\sM^\vee)
		\]
		of $\sM$ to identify $\fJ^\infty_\sM$ with $\wSym_{\sO_\sM}(\sT_\sM[1]^\vee \oplus \sT_\sM^\vee)$, the jet space $\sM_{jet}$ is then identified with the Fedosov resolution of $\sM$ as defined in \cite{liao2025formal}.
	\end{rmk}
	
	Now let $\sC$ be a quasi-dg manifold. Applying again the construction in the previous subsection to $\sO_\sC$, regarded as a quasi-dg vector bundle, we obtain a quasi-dg vector bundle
	\[
	\fJ^\infty_{\sO_\sC} = \wOmega_\sM \,\wotimes_{\sO_\sM}\, \sJ^\infty_{\sO_\sC},
	\]
	equipped with the total differential
	\[
	D_{\fJ^\infty_{\sO_\sC}} = \mathrm{d}_{\nabla^G} + L_{Q_\sM}^{\nabla^G} + \id_{\wOmega_\sM}  \otimes j^\infty(Q_\sC).
	\]
	
	Similarly, one can show that
	\begin{prop}
		The pair $\sC_{jet}\coloneqq(M, \fJ^\infty_{\sO_\sC})$ defines a fibrant quasi-dg manifold over $\sM$, called the \emph{jet space} of $\sC$. Moreover, the morphism
		\[
		\sC_{jet}=(M, \fJ^\infty_{\sO_\sC}) \xlongrightarrow{(\id_M, j^\infty)} \sC=(M, \sO_\sC)
		\]
		is a weak equivalence in $\mathbf{QdgM}$.
	\end{prop}
	
	Let $F=(f,F^\sharp)\colon \sC_1=(M_1, \sO_{\sC_1}) \rightarrow \sC_2=(M_2, \sO_{\sC_2})$ be a morphism of quasi-dg manifolds. 
	Using Propositions \ref{pullbackE}, \ref{resE}, and \ref{liftE}, one can straightforwardly verify that
	\begin{prop}
		The following diagram of filtered cdgas
		\[
		\begin{tikzcd}[column sep=80pt]
			\fJ^\infty_{\sO_{\sC_2}} = \wOmega_{\Gr^0 \sC_2} \,\wotimes_{\Gr^0 \sO_{\sC_2}}\, \sJ^\infty_{\sO_{\sC_2}}
			\arrow[r, "(\Gr^0 F)_{dR}^\sharp \otimes j^\infty(F^\sharp)"] 
			& f_* \fJ^\infty_{\sO_{\sC_1}} = f_*\left(\wOmega_{\Gr^0 \sC_1} \,\wotimes_{\Gr^0 \sO_{\sC_1}}\, \sJ^\infty_{\sO_{\sC_1}}\right) \\
			\sO_{\sC_2} \arrow[u, "j^\infty"] \arrow[r, "F^\sharp"] 
			& f_* \sO_{\sC_1} \arrow[u, "f_* j^\infty"].
		\end{tikzcd}
		\]
		commutes, where $j^\infty(F^\sharp)$ is defined in Proposition \ref{pullbackE}.
	\end{prop}
	
	Consider the assignment
	\[
	\bfJ^\infty\colon \mathbf{QdgM} \longrightarrow \mathbf{QdgM}_{\mathrm{fib}} 
	\]
	defined on objects $\sC$ by
	\[
	\bfJ^\infty(\sC) = \sC_{jet},
	\]
	and on morphisms $F=(f,F^\sharp)\colon \sC_1=(M_1, \sO_{\sC_1}) \rightarrow \sC_2=(M_2, \sO_{\sC_2})$ by
	\[
	\bfJ^\infty(F) \coloneqq \bigl(f,\, (\Gr^0 F)_{dR}^\sharp \otimes j^\infty(F^\sharp)\bigr).
	\]	
	Our previous results fit together to yield the following theorem.
	\begin{thm}\label{jetfun}
		$\bfJ^\infty$ is a well-defined functor, called the \emph{jet space functor}. Moreover, there is a natural transformation
		\[
		j^\infty\colon \bfJ^\infty \Rightarrow \id_{\mathbf{QdgM}},
		\]
		which is a component-wise weak equivalence.\footnote{Here, we regard $\bfJ^\infty$ as a functor from $\mathbf{QdgM}$ to $\mathbf{QdgM}$.} 
	\end{thm}	
	
	It then follows from the $2$-out-of-$3$ property of weak equivalences that the jet space functor $\bfJ^\infty$ detects weak equivalences. 
	
	\begin{cor}\label{main2}
		The following composition
		\[
		\mathbf{Fib}\colon  \mathbf{L_\infty Algd} \xrightarrow{\bfCEAlgd} \mathbf{QdgM}^{\mathrm{bp}} \xrightarrow{~\mathbf{J}^\infty}~ \mathbf{QdgM}^{\mathrm{bp}}_{\mathrm{fib}} \xrightarrow{\bfCESp^{-1}} \mathbf{L_\infty Sp}
		\]
		detects weak equivalences. 
	\end{cor}
	\begin{proof}
		This follows from Theorems \ref{CE1}, \ref{CE2}, and \ref{jetfun}.
	\end{proof}
	
	
	\appendix 
	
	\section{Differential operators and jets}
	\label{sec:doj}
	
	Let $\sM = (M, \sO_\sM)$ be a graded manifold.  In this appendix, we collect and prove some basic facts about differential operators and jets on $\sM$.
	
	\begin{defn}\label{DM}
		The \emph{sheaf of differential operators} $\sD_{\sM}$ on $\sM$ is defined by assigning to each open set $U \subset M$ the subalgebra of $\mathrm{End}_{\R}(\sO_{\sM}(U))$ generated by left multiplications by $\sO_{\sM}(U)$ and by vector fields $\sT_\sM(U)$. 
	\end{defn}
	$\sD_\sM$ is naturally a graded $(\sO_\sM,\sO_\sM)$-algebra. It carries a compatible descending filtration 
	\[
	\sD_\sM \supset \cdots \supset F^{-1} \sD_\sM \supset F^0 \sD_\sM \supset F^1 \sD_\sM = 0,
	\]
	given by $F^0 \sD_\sM= \sO_\sM$, and inductively,
	\[
	F^{-k} \sD_\sM = \left\{ D \in \sD_\sM \,\middle|\, [D, f] \in F^{-k+1} \sD_\sM \text{ for all } f \in \sO_\sM \right\}
	\]
	for $k > 0$, where $[D, f] = D \circ f - (-1)^{|f||D|}f \circ D$. In other words, $F^{-k}\sD_\sM$ consists of differential operators \emph{of order at most $k$}. There is a canonical isomorphism of graded $\sO_\sM$-modules
	\[
	\Gr^{-k} \sD_\sM \cong \Sym^k_{\sO_\sM} (\sT_\sM).
	\]
	This identification induces a map
	\[
	\sigma_k \colon F^{-k} \sD_\sM \longrightarrow \Sym^k_{\sO_\sM} (\sT_\sM),
	\]
	called the \emph{symbol map}. For $D \in F^{-k} \sD_\sM$, if $\sigma_k(D) \neq 0$, then $D$ is said to be \emph{of order $k$}. The element $\sigma_k(D) \in \Sym^k_{\sO_\sM} \sT_\sM$ is called the \emph{principal symbol} of $D$.
	
	$\sD_\sM$ is also a graded cocommutative $(\sO_\sM,\sO_\sM)$-coalgebra. The counit map $\epsilon\colon \sD_\sM \to \sO_\sM$ is defined by 
	\[
	\epsilon(D) \coloneqq D(1),
	\]
	and the coproduct map $\Delta\colon \sD_\sM \to \sD_\sM \otimes_{\sO_\sM} \sD_\sM$ is defined by 
	\[
	\Delta(D)(f \otimes_\R g) \coloneqq D(fg) 
	\]
	where we identify $\sD_\sM \otimes_{\sO_\sM} \sD_\sM$ with the set of bi-differential operators. Equivalently, $\Delta$ is characterized by the following formulas \cite{xu2001quantum}
	\begin{align}
		&\Delta(1) = 1 \otimes_{\sO_\sM} 1, \label{Delta1} \\
		&\Delta(X) =1 \otimes_{\sO_\sM} X + X \otimes_{\sO_\sM}1, \quad  \text{for~all~} X \in \sT_\sM, \label{DeltaX}\\
		&\Delta(D_1 \circ D_2) = \Delta(D_1) \circ \Delta(D_2), \quad \text{for~all~}D_1, D_2 \in \sD_\sM. \label{compprod}
	\end{align}
	Moreover, $\epsilon$ and $\Delta$ are compatible with the filtration on $\sD_\sM$:
	\[
	\epsilon(F^k \sD_\sM) \subset F^k \sO_\sM,
	\qquad 
	\Delta(F^k \sD_\sM)\subset \sum_{i+j=k} F^i \sD_\sM \otimes_{\sO_\sM} F^j \sD_\sM,
	\]
	where $\sO_\sM$ is equipped with the trivial filtration. Indeed, $F^\bullet \sD_\sM$ is precisely the $\sI$-adic filtration determined by the coideal $\sI \coloneqq F^0 \sD_\sM$:
	\[
	F^{-p} \sD_\sM
	=
	\left\{
	D \in \sD_\sM \;\middle|\;
	\Delta^{(p)}(D)
	\in \sum_{k=0}^p 
	\sD_\sM^{\otimes k}
	\otimes \sI
	\otimes \sD_\sM^{\otimes (p-k)}
	\right\}, 
	\]
	for all $p \in \N$,	where $\Delta^{(p)}$ denotes the $p$-fold iteration of $\Delta$.
	
	\begin{lem}
		There is a one-to-one correspondence between splittings
		\[
		\mathrm{pbw}^1 \colon \Sym^2_{\sO_\sM}(\sT_\sM) \longrightarrow F^{-2} \sD_\sM
		\]
		of the short exact sequence of graded (left) $\sO_\sM$-modules
		\[
		0 \longrightarrow F^{-1} \sD_\sM \longrightarrow F^{-2} \sD_\sM \longrightarrow \Sym^2_{\sO_\sM}(\sT_\sM) \longrightarrow 0,
		\]
		and torsion-free affine connections $\nabla$ on $\sM$.
	\end{lem}
	\begin{proof}
		The correspondence is given explicitly by the formula
		\begin{equation}\label{pbw1}
			\pbw^1(X \odot Y) = \frac{1}{2}\left((X \circ Y - \nabla_X Y) + (-1)^{|X||Y|}(Y \circ X - \nabla_Y X)\right),
		\end{equation}
		where $X \odot Y$ denotes the symmetric product between $X$ and $Y$.
	\end{proof}
	
	The formula \eqref{pbw1} extends inductively to maps
	\[
	\pbw^n \colon \Sym^{n+1}_{\sO_\sM}(\sT_\sM) \longrightarrow F^{-n-1} \sD_\sM.
	\]
	The map $\pbw^0 \colon \sT_\sM \to F^{-1} \sD_\sM \cong \sO_\sM \oplus \sT_\sM$ is the canonical splitting. For $n \ge 1$, we define
	\[
	\pbw^n(X_0 \odot X_1 \odot \cdots \odot X_n)
	=
	\frac{1}{n+1}
	\sum_{k=0}^n \epsilon_k \left(
	X_k \circ \pbw^{n-1}(X^{\{k\}})
	- \pbw^{n-1}\bigl(\nabla_{X_k}(X^{\{k\}})\bigr)
	\right),
	\]
	where $\nabla$ is a torsion-free affine connection on $\sM$, $X^{\{k\}} = X_0 \odot \cdots \odot X_{k-1} \odot X_{k+1} \odot \cdots \odot X_n$, and $\epsilon_k = (-1)^{|X_k| (|X_0| + \cdots + |X_{k-1}|)}$.

    \begin{prop}[Theorem 3.3 in \cite{liao2019formal}]\label{pbw}
    	The map
    	\[
    	\pbw\colon \Sym_{\sO_\sM} (\sT_\sM) \to \sD_\sM,
    	\] 
    	induced by a torsion-free affine connection $\nabla$ on $\sM$, is a filtered graded cocommutative $\sO_\sM$-coalgebra isomorphism.
    \end{prop}
   The converse of the above statement does not hold.
	
	\begin{defn}
		The \emph{sheaf of $k$-jets} $\sJ^k_\sM$ on $\sM$ is defined as 
		\[
		\sJ^k_\sM \coloneqq F^{-k} \sD_\sM^\vee = \Hom_{\sO_\sM}(F^{-k} \sD_\sM, \sO_\sM),
		\]
		where $\Hom_{\sO_\sM}$ denotes the internal Hom in the category of graded left $\sO_\sM$-modules. 
		
		$\{\sJ^k_\sM\}_{k \geq 0}$ together with the canonical epimorphisms $\{\sJ^j_\sM \to \sJ^i_\sM\}_{j \geq i \geq 0}$ forms an inverse system. The \emph{sheaf of $\infty$-jets} $\sJ^\infty_\sM$ on $\sM$ is defined as the inverse limit
		\[
		\sJ^\infty_\sM \coloneqq \varprojlim \sJ^k_\sM =  \Hom_{\sO_\sM}(\varinjlim F^{-k} \sD_\sM, \sO_\sM) = \sD_\sM^\vee.
		\]
	\end{defn}
	
	By definition, $\sJ^\infty_\sM$ carries an $(\sO_\sM, \sO_\sM)$-bimodule structure induced by the one on $\sD_\sM$. It also carries a descending filtration
	\[
	\sJ^\infty_\sM = F^0 \sJ^\infty_\sM \supset F^1 \sJ^\infty_\sM \supset \cdots,
	\]
	induced by the filtration on $\sD_\sM$. More precisely, 
	\[
	F^k \sJ^\infty_\sM = \{ J \in \sJ^\infty_\sM\colon J(F^l \sD_\sM) \subset F^{k+l} \sO_\sM\}.
	\] 
	There are canonical graded $\sO_\sM$-module isomorphisms:
	\[
	\Gr^k \sJ^\infty_\sM \cong \Sym^k_{\sO_\sM}(\sT_\sM^\vee).
	\]
	
	Moreover, $ \sJ^\infty_\sM $ is a filtered graded commutative $(\sO_\sM, \sO_\sM)$-algebra, whose product map $\sJ^\infty_\sM \,\wotimes_{\sO_\sM}\, \sJ^\infty_\sM \to \sJ^\infty_\sM$ is defined as the dual of the coproduct $ \Delta $ on $ \sD_\sM $
	\[
    (J_1 \cdot J_2)(D) = (J_1 \,\wotimes_{\sO_\sM}\, J_2) (\Delta D),
	\]
	where $\wotimes_{\sO_\sM}$ denotes the completed tensor product over $\sO_\sM$, and we use the identification
	\[
	\Hom_{\sO_\sM}(\sD_\sM \otimes_{\sO_\sM} \sD_\sM, \sO_\sM) = \sJ^\infty_\sM \,\wotimes_{\sO_\sM}\, \sJ^\infty_\sM.
	\]
	Likewise, $F^\bullet \sJ^\infty_\sM$ is precisely the $\sI$-adic filtration determined by the ideal $\sI \coloneqq F^1 \sJ^\infty_\sM$:
	\[
	F^p \sJ^\infty_\sM = \sI^p,
	\]
	for all $p \in \N$.
	
	\begin{defn}
		A \emph{formal exponential map} on $\sM$ is a filtered graded commutative $\sO_\sM$-algebra isomorphism
		\[
		\fem\colon \sJ^\infty_\sM \xlongrightarrow{\cong} \widehat{\Sym}_{\sO_\sM} (\sT_\sM^\vee).
		\]
	\end{defn}
	
	Each formal exponential map on $\sM$ corresponds to a filtered graded cocommutative $\sO_\sM$-coalgebra isomorphism 
	\[
	\Sym_{\sO_\sM} (\sT_\sM) \xrightarrow{\cong} \sD_\sM,
	\]
    and vice versa. In particular, each torsion free affine connection on $\sM$ determines a formal exponential map on $\sM$.
	
	\begin{rmk}
		A formal exponential map on an ordinary manifold $M$ can also be defined as an equivalence class $[\phi]$ of generalized exponential maps; we follow the exposition of \cites{bonechi2012poisson,cattaneo2019globalization}.
		
		Recall that a \emph{generalized exponential map} on $M$ is a smooth map  
		\[
		\phi\colon U \to M, \quad (x, y \in U_x) \mapsto \phi_x(y),
		\]
		where $U$ is an open neighborhood of the zero section in $TM$, such that for all $x \in M$:
		\[
		\phi_x(0) = x, \quad d_y \phi_x|_{y=0} = \mathrm{id}.
		\]
		Two generalized exponential maps are said to be equivalent if, at each point of $M$, all their partial derivatives in the vertical directions agree at the zero section. 
		
		Given such an equivalence class $[\phi]$, one obtains a morphism
		\[
		\fem\colon J^\infty_M \to \widehat{\Sym}_{C^\infty_M} (T_M^\vee)
		\]
		by setting
		\[
		\fem([f])(X_1, \dots, X_n)(x) = \frac{d}{dt_1}\bigg|_{t_1=0} \cdots \frac{d}{dt_n}\bigg|_{t_n=0} f_x(\phi_x(t_1 X_1(x) + \cdots + t_n X_n(x))),
		\]
		where $[f]$ is a section of $\sJ^\infty_M$, and $f_x$ is a local smooth function representing $[f]$ at $x \in M$.\footnote{Such an $f$ exists by Borel's lemma.}
		It is straightforward to check that $\fem$ is a filtered graded commutative algebra isomorphism over $C^\infty_M$.
	\end{rmk}

	\begin{defn}
		The \emph{$k$-th (resp. infinite) jet prolongation} is the $\R$-linear morphism
		\[
		j^k\colon \sO_\sM \to \sJ^k_\sM \quad \text{(resp. } j^\infty\colon \sO_\sM \to \sJ^\infty_\sM\text{)},
		\]
		defined by the evaluation map
		\[
		f \mapsto (D \mapsto (-1)^{|D||f|} D(f)),
		\]
		where $D$ ranges over differential operators of order $\leq k$ (resp. arbitrary order).
	\end{defn}
	
	\begin{lem}\label{prol}
		$j^\infty$ is a filtered graded commutative algebra morphism. Moreover, it induces the canonical right $\sO_\sM$-module structure of $\sJ^\infty_\sM$, i.e.,
		\begin{align}\label{jdf}
			(Jj^\infty(f))(D)=(-1)^{|D||f|}J(D \circ f),
		\end{align}
		where $J$ is an $\infty$-jet, $D$ is a differential operator, and $f$ is a function on $\sM$.
	\end{lem}
	\begin{proof}
		Let $g$ be a function and $X$ a vector field over $\sM$. Using \eqref{Delta1}, we obtain
		\[
		(Jj^\infty(f))(g)=J \otimes j^\infty(f)(\Delta g) = J(1)j^\infty(f)(g)= J(1)fg = (-1)^{|g||f|}J(gf).
		\]
		Now assume that $D$ satisfies \eqref{jdf} for all $\infty$-jets and functions. Using \eqref{DeltaX} and \eqref{compprod}, we compute
		\begin{align*}
			(Jj^\infty(f))(D \circ X)&= J \otimes_\R j^\infty(f)(\Delta D \circ (1 \otimes  X + X \otimes 1) )\\
			&=(-1)^{|X|(|D|+|f|)}\left((Jj^\infty(X(f)))(D) + (J_Xj^\infty(f))(D)\right) \\
			&= (-1)^{(|D|+|X|)|f|} (J(D \circ X(f))+(-1)^{|X||f|}D \circ f \circ X) \\
			&= (-1)^{(|D|+|X|)|f|}J(D\circ X \circ f),
		\end{align*}
		where $J_X(\cdot)\coloneqq J(\cdot \circ X) \in \sJ^\infty_\sM$. By induction, \eqref{jdf} holds for all differential operators. It follows that
		\[
		(j^\infty(g)j^\infty(f))(D)=(-1)^{|D||f|}j^\infty(g)(D\circ f) = (-1)^{|D||f|+|g|(|D|+|f|)}D(fg) = j^\infty(gf)(D).
		\]
		This completes the proof.
	\end{proof}
	
	It is useful to give a local description of $j^\infty$. Let $\{x^\mu\}$ be local coordinates on $\sM$, and let $\{y^\nu\}$ denote the corresponding fiber coordinates on $\sT_\sM^\vee$. Then, for a local function $f(x)$,
	\begin{equation}\label{taylor}
		j^\infty(f)(x,y)
		=
		\sum_{\alpha}
		\frac{1}{\alpha!}\,\frac{\partial^\alpha f(x)}{\partial x^\alpha}\, y^\alpha,
	\end{equation}
	where $\alpha = (k_1,\dots,k_n; k_{n+1},\dots,k_{n+m}) \in \mathbb{N}^n \times \{0,1\}^m$ is a multi-index, with $n|m$ the superdimension of $\sM$. (We assume that the coordinates $x^1,\dots,x^n$ are even, while $x^{n+1},\dots,x^{n+m}$ are odd.) We set
	\[
	\alpha! = k_1!\cdots k_n!,
	\qquad
	\frac{\partial^\alpha}{\partial x^\alpha}
	=
	\frac{\partial^{k_{n+m}}}{(\partial x^{n+m})^{k_{n+m}}}
	\circ \cdots \circ
	\frac{\partial^{k_1}}{(\partial x^1)^{k_1}},
	\qquad
	y^\alpha = (y^1)^{k_1}\cdots (y^{n+m})^{k_{n+m}}.
	\]
	
	In other words, under the local identification of $\sJ^\infty_\sM$ with $\widehat{\Sym}_{\sO_\sM}(\sT_\sM^\vee)$ induced by a trivial local affine connection, $j^\infty(f)$ is the Taylor expansion of $f$.
	
	\begin{lem}\label{prolongdense}
		The graded $\sO_\sM$-module $\sJ^\infty_\sM$ is generated by $j^\infty(\sO_\sM)$ over $\sO_\sM$.
	\end{lem}
	
	\begin{proof}
		Since everything is smooth, it suffices to verify the statement locally. By \eqref{taylor}, all monomials $y^\alpha$ lie in the local $\sO_\sM$-span of $j^\infty(\sO_\sM)$, since
		\[
		j^\infty(x^\alpha) = y^\alpha + \cdots.
		\]
		This completes the proof.
	\end{proof}
	
	Let $f\colon \sM \rightarrow \sN$ be a morphism of graded manifolds.
	
	\begin{prop}\label{pullback}
		There exists a unique filtered graded commutative algebra morphism
		\[
		j^\infty(f^\sharp)\colon \sJ^\infty_\sN \longrightarrow f_* \sJ^\infty_\sM,
		\]
		such that the following diagram 
		\[
		\begin{tikzcd}[column sep=50 pt]
			\sJ^\infty_\sN \arrow[r, "j^\infty(f^\sharp)"] & f_*\sJ^\infty_\sM \\
			\sO_\sN \arrow[u, "j^\infty"] \arrow[r, "f^\sharp"] & f_*\sO_\sM \arrow[u, "f_* j^\infty"]
		\end{tikzcd}
		\]
		commutes.
	\end{prop}
	
	\begin{proof}
		Since the image $j^\infty(\sO_\sN)$ generates $\sJ^\infty_\sN$, we define $j^\infty(f^\sharp)$ simply by setting
		\[
		j^\infty(f^\sharp)(g) = f^\sharp(g), 
		\qquad 
		j^\infty(f^\sharp)\bigl(j^\infty(g)\bigr) = j^\infty\bigl(f^\sharp(g)\bigr),
		\]
		for all $g \in \sO_\sN$. To check well-definedness, it suffices to show that
		\[
		\sum_{i \in I} f^\sharp(g_i)\, j^\infty\bigl(f^\sharp(h_i)\bigr) = 0
		\quad \text{whenever} \quad
		\sum_{i \in I} g_i\, j^\infty(h_i) = 0.
		\]
		This follows from the compatibility of $f^\sharp$ with derivations. Indeed, locally,
		\[
		f^\sharp(g)\, \frac{\partial}{\partial x_\sM^\mu} f^\sharp(h_i)
		=
		f^\sharp(g)\, \frac{\partial f^\sharp(x_\sN^\nu)}{\partial x_\sM^\mu}\,
		f^\sharp\!\left(\frac{\partial h_i}{\partial x_\sN^\nu}\right)
		=
		\pm \frac{\partial f^\sharp(x_\sN^\nu)}{\partial x_\sM^\mu}\,
		f^\sharp\!\left(g\, \frac{\partial h_i}{\partial x_\sN^\nu}\right),
		\]
		which implies the claim. Here, $x_\sM^\mu$ and $x_\sN^\nu$ are coordinates on $\sM$ and $\sN$, respectively.
	\end{proof}
	
	\begin{rmk}
		Let $y_\sM^\mu$ and $y_\sN^\nu$ denote the fiber coordinates on $\sT_\sM^\vee$ and $\sT_\sN^\vee$, respectively. By Proposition \ref{pullback}, one obtains the local expression
		\[
		j^\infty(f^\sharp)(y_\sN^\mu)
		=
		j^\infty\bigl(f^\sharp(x_\sN^\mu)\bigr) - f^\sharp(x_\sN^\mu)
		=
		\sum_{|\alpha| > 0} \frac{y_\sM^\alpha}{\alpha!}
		\frac{\partial^\alpha f^\sharp(x_\sN^\mu)}{\partial x_\sM^\alpha}.
		\]
	\end{rmk}
	
	 Let $\sE$ be a filtered graded vector bundle on $\sM$.	 
	 \begin{defn}
	 	The \emph{sheaf of differential operators} on $\sE$ is defined as
	 	\[
	 	\sD_\sE\coloneqq  \sE \,\wotimes_{\sO_\sM}\, \sD_\sM \,\wotimes_{\sO_\sM}\, \sE^\vee.
	 	\]
	 \end{defn}
	 
	 Let $k \in \N$ and $p \in \Z$. Sections of 
	 \[
	 F_O^{-k} \sD_\sE\coloneqq\sE \,\wotimes_{\sO_\sM}\, F^{-k}\sD_\sM \,\wotimes_{\sO_\sM}\, \sE^\vee 
	 \]
	 are differential operators \emph{of order at most $k$} on $\sE$, while sections of 
	 \[
	 F_I^p \sD_\sE \coloneqq \{D_\sE \in \sD_\sE\colon D_\sE(F^\bullet \sE)
	 \footnote{Let $D_\sE = e \otimes D \otimes e^\vee \in \sD_\sE$ and $s \in \sE$. Then $D_\sE$ acts on $s$ by
	 	\[
	 	D_\sE(s) =  e D(e^\vee(s)).
	 	\]} 
	 \subset F^{\bullet + p} \sE\}
	 \]
	 are differential operators \emph{of filtration degree at most $p$} on $\sE$. In particular, $F_I^0 \sD_\sE$ consists of filtration-preserving differential operators.
	 
	 $\sD_\sE$ is a filtered graded $(\sO_\sM, \sO_\sM)$-algebra, whose product is given by composition of differential operators. Moreover,  the default filtration $F^\bullet \sD_\sE$ on $\sD_\sE$ is the same as the convolution of the two filtrations $F_O^\bullet \sD_\sE$ and $F_I^\bullet \sD_\sE$:
	 \[
	 F^p \sD_\sE \coloneqq \sum_{i+j = p} F_I^i \sD_\sE \circ F_O^j \sD_\sE.
	 \]
	 There is a canonical isomorphism of graded $\sO_\sM$-modules
	 \[
	 \Gr^{p} \sD_\sE \cong \bigoplus_{q \in \N} \Sym^q_{\sO_\sM} (\sT_\sM) \otimes_{\sO_\sM} \Gr^{p+q} \End_\sE,
	 \]
	 where $\End_\sE \coloneqq \wHom_{\sO_\sM}(\sE, \sE)$ and $\wHom_{\sO_\sM}$ denotes the completed internal Hom in the category of filtered graded left $\sO_\sM$-modules. This identification induces a map
	 \[
	 \sigma_k^p \colon F_O^{-k} \sD_\sE \cap F_I^p \sD_\sE \subset F^{p-k} \sD_\sE \longrightarrow \Sym^k_{\sO_\sM} (\sT_\sM) \otimes_{\sO_\sM} \Gr^p \End_\sE,
	 \]
	 called the \emph{symbol map}. For $D_\sE \in F_O^{-k} \sD_\sE \cap F_I^p \sD_\sE$, if $\sigma_k^p(D_\sE) \neq 0$, then $D_\sE$ is said to be \emph{of order $k$ and filtration degree $p$}. The element 
	 \[
	 \sigma_k^p(D_\sE) \in \Sym^k_{\sO_\sM} (\sT_\sM) \otimes_{\sO_\sM} \Gr^p \End_\sE
	 \]
	 is called the \emph{principal symbol} of $D_\sE$.	
	 
	 \begin{defn}
	 	The \emph{$\infty$-jet bundle} of $\sE$ is the filtered graded vector bundle
	 	\[
	 	\sJ^\infty_\sE \coloneqq \sJ^\infty_{\sM}\,\wotimes_{\sO_\sM}\,\sE,
	 	\]
	 	where $\sJ^\infty_{\sM}$ is viewed as a right $\sO_\sM$-module.
	 \end{defn}
	 To see that $\sJ^\infty_\sE$ is well-defined as a filtered graded vector bundle, choose a splitting
	 \[
	 \sigma\colon\sE \longrightarrow \sJ^\infty_\sE
	 \]
	 of the short exact sequence of graded left $\sO_\sM$-modules
	 \[
	 0 \longrightarrow F^1 \sJ^\infty_\sM \,\wotimes_{\sO_\sM}\, \sE \longrightarrow \sJ^\infty_\sM \,\wotimes_{\sO_\sM}\, \sE \longrightarrow \sE \longrightarrow 0.
	 \]
	 Consider the morphism of filtered graded left $\sO_\sM$-modules
	 \begin{align*}
	 	\Phi^\sigma\colon \sE \,\wotimes_{\sO_\sM}\, \sJ^\infty_\sM &\longrightarrow \sJ^\infty_\sM \,\wotimes_{\sO_\sM}\, \sE \\
	 	s \otimes j &\longmapsto (-1)^{|j||s|}  j \cdot \sigma(s).
	 \end{align*}
	 Its associated graded morphism identifies with
	 \begin{align*}
	 	\Gr\, \Phi^\sigma\colon \Gr\, \sE \otimes_{\sO_\sM}  \Sym_{\sO_\sM}(\sT_\sM^\vee)
	 	&\longrightarrow \Sym_{\sO_\sM}(\sT_\sM^\vee) \otimes_{\sO_\sM} \Gr\, \sE \\
	 	s \otimes \alpha &\longmapsto (-1)^{|\alpha||s|} \alpha \otimes s,
	 \end{align*}
	 which is obviously an isomorphism. Since the filtrations are complete and exhaustive, it follows that $\Phi^\sigma$ itself is an isomorphism. Therefore, $\sJ^\infty_\sE$ is (non-canonically) isomorphic, as a filtered graded left $\sO_\sM$-module, to the filtered graded vector bundle $\sE \,\wotimes_{\sO_\sM}\, \sJ^\infty_\sM$.

	\begin{lem}
		Let $\{(U,\varphi_U)\}$ be a family of local trivializations of $\sE$. The local $\R$-linear morphisms $j^\infty|_U\colon \sE(U) \to \sJ^\infty_\sE(U)$, defined via the commutative diagram
		\[
		\begin{tikzcd}[column sep=50 pt]
			\sE(U) \arrow[d, "\varphi_U"'] \arrow[r, "j^\infty|_U"] 
			& \sJ^\infty_\sM(U) \,\wotimes_{\sO_\sM(U)}\, \sE(U) \\
			\sO_\sM(U) \,\wotimes_\R\, F \arrow[r, "{~j^\infty|_U \otimes 1~}"'] 
			& \sJ^\infty_\sM(U) \,\wotimes_{\sO_\sM(U)}\, (\sO_\sM(U) \, \wotimes_\R \, F) \arrow[u, "{1 \,\wotimes\, \varphi_U^{-1}}"],
		\end{tikzcd}
		\]
		where $F$ denotes the fiber of $\sE$. These local morphisms glue to a well-defined, filtration-preserving $\mathbb{R}$-linear map of degree $0$
		\[
		j^\infty \colon \sE \to \sJ^\infty_\sE,
		\]
		called the \emph{infinite jet prolongation} of $\sE$. 
	\end{lem}
	\begin{proof}
		Without loss of generality, we assume that $\sE$ is a graded vector bundle. 
		
		Let $\{e_i\}$ be a basis of $F$. The induced sections $s_i \coloneqq \varphi_U^{-1}(1 \otimes e_i)$ form a basis for the free $\sO_\sM(U)$-module $\sE(U)$. For any local section $s = f^i s_i \in \sE(U)$, where $f^i \in \sO_\sM(U)$, the infinite jet prolongation acts by
		\[
		j^\infty|_U(s) = j^\infty|_U(f^i) s_i.
		\]
		
		Now consider another trivialization $(U', \varphi_{U'})$ with $U \cap U' \neq \emptyset$. Let $\{s'_j \coloneqq \varphi_{U'}^{-1}(1 \otimes e_j)\}$ be the corresponding basis of $\sE(U')$, and let $s_i|_{U \cap U'} = g_i^j s'_j|_{U \cap U'}$ for transition functions $g_i^j \in \sO_\sM(U \cap U')$. On the intersection $U \cap U'$, we verify the compatibility:
		\begin{align*}
			j^\infty|_{U \cap U'}(s|_{U \cap U'}) 
			&= j^\infty|_{U \cap U'}(f^i|_{U \cap U'})  s_i|_{U \cap U'} \\
			&= j^\infty|_{U \cap U'}(f^i|_{U \cap U'})  g_i^j  s'_j|_{U \cap U'} \\
			&= j^\infty|_{U \cap U'}(f^i|_{U \cap U'} g_i^j)  s'_j|_{U \cap U'} \\
			&= j^\infty|_{U \cap U'}((f')^j|_{U \cap U'})  s'_j|_{U \cap U'},
		\end{align*}
		where $(f')^j = f^i g_i^j$. We use Lemma \ref{prol} to pass to the third line.
	\end{proof}
	
	\begin{lem}\label{prolongdenseE}
		The filtered graded $\sO_\sM$-module $\sJ^\infty_\sE$ is generated by $j^\infty(\sE)$ over $\sO_\sM$.
	\end{lem}

	Let $f\colon \sM_1 \rightarrow \sM_2$ be a morphism of graded manifolds. Let $\sE_1$ and $\sE_2$ be filtered graded vector bundles over $\sM_1$ and $\sM_2$, respectively. Let 
	\[
	F^\sharp\colon \sE_2 \rightarrow f_* \sE_1
	\]
	be a morphism of filtered graded $\sO_{\sM_2}$-modules. (Note that the pair $(f, F^\sharp)$ does not define a morphism of filtered graded vector bundles, since $F^\sharp$ goes in the opposite direction to the usual bundle map.)
	
	\begin{prop}\label{pullbackE}
		There exists a unique filtered graded $\sO_{\sM_2}$-module morphism
		\[
		j^\infty(F^\sharp)\colon \sJ^\infty_{\sE_2} \rightarrow f_* \sJ^\infty_{\sE_1},
		\]
		such that the following diagram
		\[
		\begin{tikzcd}[column sep=50 pt]
			\sJ^\infty_{\sE_2} \arrow[r, "j^\infty(F^\sharp)"] & f_*\sJ^\infty_{\sE_1} \\
			\sE_2 \arrow[u, "j^\infty"] \arrow[r, "F^\sharp"] & f_*\sE_1 \arrow[u, "f_*j^\infty"]
		\end{tikzcd}
		\]
		commutes.
	\end{prop}
	
	The proofs of Lemma \ref{prolongdenseE} and Proposition \ref{pullbackE} follow the same arguments as in the proofs of Lemma \ref{prolongdense} and Proposition \ref{pullback}, and are therefore omitted.
	
	\section{De Rham complexes of \D-modules}
	\label{sec: Dmod}
	
	Let $\sM = (M, \sO_\sM)$ be a graded manifold. In this appendix, we consider (left) $D$-modules on $\sM$ and their de Rham complexes.
	
	\begin{defn}
		A \emph{(left) $D$-module} on $\sM$ is a filtered graded vector bundle $\sE$ on $\sM$, together with a left action of the filtered graded $(\sO_\sM,\sO_\sM)$-algebra $\sD_\sM$ on $\sE$: 
		\[
		F^p \sD_\sM \otimes_{\sO_\sM} F^q \sE \longrightarrow F^{p+q} \sE.
		\]
	\end{defn}
	Let $\sE$ be a $D$-module. Then for every vector field $X \in \sT_\sM$, one has
	\[
	X \cdot F^\bullet \sE \subset F^{\bullet-1} \sE.
	\]
	
	\begin{defn}
		Let $\sE$ be a filtered graded vector bundle.
		A \emph{connection} on  $\sE$ is an $\R$-linear map
		\[
		\nabla\colon F^\bullet \sE \to  \sT_\sM^\vee \otimes_{\sO_\sM} F^{\bullet-1} \sE
		\]
		such that
		\[
		\nabla (f s) = \mathrm{d}f \otimes s + (-1)^{|f|} f \nabla s.
		\]
	\end{defn}
	We use the standard notation $\nabla_X$ for the composition $\sE \xrightarrow{\nabla} \sT_\sM^\vee  \otimes_{\sO_\sM} \sE \xrightarrow{\iota_X \otimes \mathrm{id}} \sE$.
	\begin{rmk}
		If $\sE$ is a graded vector bundle, regarded as a filtered graded vector bundle with the trivial filtration, then all left $\sD_\sM$-actions and connections on $\sE$ are automatically compatible with the trivial filtration.
	\end{rmk}
	
	\begin{exmp}
		A connection on the tangent bundle $\sT_\sM$ is called an \emph{affine connection}. Such a connection $\nabla$ is called \emph{torsion-free} if
		\[
		\nabla_X Y - (-1)^{|X||Y|}\nabla_Y X = [X,Y].
		\]
		Using partition of unity, one can show that there always exists a torsion-free affine connection.
	\end{exmp}
	
	Recall that the sheaf of completed forms on $\sM$ is 
	\[
	\wOmega_\sM \coloneqq \wSym_{\sO_\sM}(\sT_\sM[1]^\vee),
	\]
	where $ \widehat{\Sym}_{\sO_\sM}$ denotes the completed symmetric power over $\sO_\sM$. We define the \emph{exterior covariant derivative} of $\nabla$ as the $\R$-linear map
	\begin{align*}
		\mathrm{d}_\nabla\colon \wOmega_{\sM} \,\wotimes_{\sO_{\sM}}\,  \sE &\rightarrow \wOmega_{\sM}\,\wotimes_{\sO_{\sM}}\,  \sE \\
		\alpha \otimes e &\mapsto \mathrm{d}_{dR} \alpha \otimes e + (-1)^{|\alpha|} \alpha \wedge \nabla e.
	\end{align*}
	By definition, $\mathrm{d}_\nabla$ preserves the filtration on $\wOmega_{\sM} \wotimes_{\sO_{\sM}}\, \sE$.
	
	\begin{defn}
		A connection $\nabla$ on $\sE$ is called \emph{flat} if $\mathrm{d}_\nabla \circ \mathrm{d}_\nabla = 0$. 
	\end{defn}
	
	It is straightforward to show that
	\begin{lem}\label{flatD}
		There is a one-to-one correspondence between flat connections on $\sE$ and $D$-module structures on $\sE$.
	\end{lem}
	Let $\sE$ be a $D$-module. Let $\nabla$ be the corresponding flat connection on $\sE$. Explicitly,
	\[
	\nabla_X (s) = X \cdot s.
	\]
	\begin{defn}
		We call the filtered dg module 
		\[
		dR(\sE)\coloneqq(\wOmega_{\sM} \,\wotimes_{\sO_\sM}\, \sE, \mathrm{d}_{\nabla})
		\]
		over $(\wOmega_\sM, \mathrm{d}_{dR})$ the \emph{de Rham complex} of $\sE$.
	\end{defn}
	\begin{defn}\label{cld}
		Let $X$ be a vector field on $\sM$.
		The \emph{covariant Lie derivative} on $dR(\sE)$ 
		\[
		L^\nabla_X\colon F^\bullet dR(\sE) \longrightarrow F^{\bullet - 1} dR(\sE)
		\]
		along $X$ is defined by
		\[
		L^\nabla_X (\alpha \otimes s) = L_X(\alpha) \otimes s + (-1)^{|X||\alpha|} \alpha \otimes \nabla_X(s),
		\]
		for $\alpha \in \wOmega_\sM$ and $s \in \sE$, where $L_X(\alpha)$ is the Lie derivative of $\alpha$.
	\end{defn}
	\begin{rmk}
		It is straightforward to verify that 
		\[
		L^\nabla_X = [\iota_X \otimes \id,  \mathrm{d}_\nabla],
		\]
		where $\iota_X\colon \wOmega_\sM \to \wOmega_\sM$ is the contraction by $X$. Moreover, since $\nabla$ is flat, one has 
		\[
		[\mathrm{d}_\nabla, L_X^\nabla] = 0, \quad 
		[L_X^\nabla, L_Y^\nabla] = L_{[X,Y]}^\nabla.
		\]
	\end{rmk}
	
	Let $\sJ^\infty_\sM$ be the sheaf of infinite jets on $\sM$. 
	
	\begin{defn}
		The \emph{Grothendieck connection} on $\sJ^\infty_\sM$, denoted by $\nabla^G$, is the flat connection defined by dualizing the canonical left $\sD_\sM$-module structure of $\sD_\sM$.  Explicitly, 
		\begin{align*}
			(\nabla^G_X J)(D) = X(J(D)) - (-1)^{|X||J|} J(X \circ D).
		\end{align*}
	\end{defn}
	
	For any function $f$ on $\sM$, we have
	\begin{align*}
		\nabla^G_X(j^\infty(f))(D)
		&= (-1)^{|X||f|}\bigl( X(D(f)) - (\nabla^G_X D)(f) \bigr) \\
		&= (-1)^{|X||f|}\bigl( X(D(f)) - (X \circ D)(f) \bigr) \\
		&= 0.
	\end{align*}
	We use this observation to determine a local formula for $\nabla^G$. Indeed,
	\[
	\nabla^G\bigl(j^\infty(x^\mu)\bigr)
	= \nabla^G(x^\mu + y^\mu)
	= dx^\mu+ \nabla^G(y^\mu) = 0.
	\]
	Since $\nabla^G$ is a derivation, it follows that it must take the local form
	\begin{equation}\label{localG}
		\nabla^G = dx^\mu \frac{\partial}{\partial x^\mu} - dx^\mu \frac{\partial}{\partial y^\mu}.
	\end{equation}
	
	\begin{lem}
		The associated $\mathbb{N}$-graded operator of the differential $\mathrm{d}_{\nabla^G}$ on $dR(\sJ^\infty_\sM)$ is minus the \emph{universal Koszul differential}
		\[
		\deltaK\colon \Omega^k_\sM \otimes_{\sO_\sM} \Sym_{\sO_\sM}^l (\sT_\sM^\vee) 
		\longrightarrow 
		\Omega^{k+1}_\sM \otimes_{\sO_\sM} \Sym_{\sO_\sM}^{l-1} (\sT_\sM^\vee).
		\]
	\end{lem}
	
	\begin{proof}
		The associated $\mathbb{N}$-graded algebra of $dR(\sJ^\infty_\sM)$ is given by the universal Koszul algebra of the graded $\sO_\sM$-module $\sT_\sM$:
		\[
		\Gr\, dR(\sJ^\infty_\sM)
		\cong
		\Omega_\sM \otimes_{\sO_\sM} \Sym_{\sO_\sM}(\sT_\sM^\vee).
		\]
		
		Note that $dx^\mu \wedge$ increases the filtration degree by $1$, while $\frac{\partial}{\partial y^\mu}$ decreases it by $1$. Using \eqref{localG}, we obtain the local formula
		\[
		-\Gr\, \mathrm{d}_{\nabla^G}	= dx^\mu \frac{\partial}{\partial y^\mu},
		\]
		which is independent of local coordinates and is exactly the universal Koszul differential.
	\end{proof}
	
	Using \eqref{DeltaX} and \eqref{compprod}, one can easily check that $\nabla^G_X$ is a derivation of $\sJ^\infty_\sM$. It follows that $\mathrm{d}_{\nabla^G}$, as well as any covariant Lie derivative on $dR(\sJ^\infty_\sM)$, is also a derivation. Consequently, $dR(\sJ^\infty_\sM)$ is a filtered cdga over $(\wOmega_\sM, \mathrm{d}_{dR})$.
	
	\begin{prop}\label{res}
		The infinite jet prolongation
		\[
		j^\infty\colon \sO_{\sM} \longrightarrow dR(\sJ^\infty_\sM)
		\]
		is a filtered quasi-isomorphism of cdgas. More precisely, a stronger statement holds: on the $E_0$ pages of the associated spectral sequences,
		\[
		\Gr\, j^\infty\colon \sO_\sM \longrightarrow \Gr\, dR(\sJ^\infty_\sM)
		\]
		is already a quasi-isomorphism of cdgas.
	\end{prop}
	\begin{proof}
		It is straightforward to verify that $\Gr\, j^\infty$ is given by the canonical inclusion
		\[
		\sO_\sM \hookrightarrow \Omega_\sM \otimes_{\sO_\sM} \Sym_{\sO_\sM}(\sT_\sM^\vee).
		\]
		The statement then follows from the fact that the cohomology of the universal Koszul algebra of a finitely generated projective $\sO_\sM$-module is the base $\sO_\sM$.
	\end{proof}
	
	\begin{prop}\label{lift}
		Every differential operator $D$ of order $k$ on $\sM$ uniquely lifts to a morphism of $D$-modules
		\[
		j^\infty(D)\colon F^\bullet \sJ^\infty_\sM \longrightarrow F^{\bullet-k} \sJ^\infty_\sM,
		\]
		whose associated $\N$-graded operator
		\[
		\Gr\, j^\infty(D)\colon \Sym^\bullet_{\sO_\sM}(\sT_\sM^\vee) \longrightarrow   \Sym^{\bullet-k}_{\sO_\sM}(\sT_\sM^\vee)
		\]
		is given by contraction with the principal symbol $\sigma_k(D) \in  \Sym^k_{\sO_\sM}(\sT_\sM)$ of $D$.	
		
		Moreover, the following diagram
		\[
		\begin{tikzcd}
			\sJ^\infty_\sM \arrow[r, "j^\infty(D)"] & \sJ^\infty_\sM \\
			\sO_\sM \arrow[u, "j^\infty"] \arrow[r, "D"] & \sO_\sM \arrow[u, "j^\infty"]
		\end{tikzcd}
		\]
		commutes.
	\end{prop}
	\begin{proof}
		Since the image $j^\infty(\sO_\sM)$ generates $\sJ^\infty_\sM$, we define $j^\infty(D)$ simply by setting
		\[
		j^\infty(D)\bigl(j^\infty(f)\bigr) \coloneqq j^\infty\bigl(D(f)\bigr).
		\]
		We need to verify that this is well-defined, i.e.,
		\[
		\sum_{i \in I} g_i\, j^\infty\bigl(D(f_i)\bigr) = 0
		\quad \text{whenever} \quad
		\sum_{i \in I} g_i\, j^\infty(f_i) = 0,
		\]
		where $I$ is some finite index set.
		
		Recall that $j^\infty(f_i) = \sum_\alpha \frac{1}{\alpha!}\, \frac{\partial^\alpha}{\partial x^\alpha}(f_i)\, y^\alpha$. It follows that
		\[
		\sum_{i \in I} g_i\, \frac{\partial^\alpha}{\partial x^\alpha}(f_i) = 0
		\]
		for all multi-indices $\alpha$. Therefore,
		\begin{align*}
			\sum_{i \in I} g_i\, j^\infty\bigl(D(f_i)\bigr)
			&= \sum_\alpha \frac{y^\alpha}{\alpha!} \sum_{i \in I} g_i\, \frac{\partial^\alpha}{\partial x^\alpha}\bigl(D(f_i)\bigr) \\
			&= \sum_{|\beta| \leq k} \sum_\alpha \frac{y^\alpha}{\alpha!} \sum_{i \in I} g_i\, \frac{\partial^\alpha}{\partial x^\alpha}\bigl(D_\beta \partial_\beta(f_i)\bigr) \\
			&= 0,
		\end{align*}
		where $D_\beta$ is some local function. In local coordinates, $j^\infty(D)$ is given by
		\[
		j^\infty(D) = \sum_{|\beta| \leq k}  \sum_{\alpha} \frac{\partial^\alpha D_\beta}{\partial x^\alpha}  \frac{y^\alpha}{\alpha!} \frac{\partial^\beta}{\partial y^\beta}.
		\]
		It follows that $j^\infty(D)$ lowers the filtration degree by $k$ and
		\[
		\Gr\,j^\infty(D) =  \sum_{|\beta| = k}  D_\beta \frac{\partial^\beta}{\partial y^\beta},
		\]
		which is exactly contraction with $\sigma_k(D)$.
		
		To see that $j^\infty(D)$ is a morphism of $D$-modules, we compute
		\[
		j^\infty(D)(\nabla^G_X(g j^\infty(f))) = j^\infty(D)(X(g) j^\infty(f)) = \pm X(g)j^\infty(D(f)) = \pm \nabla^G_X(g j^\infty(D(f))),
		\]
		which completes the proof.
	\end{proof}
	
	\begin{cor}
		Let $X$ be a vector field on $\sM$, viewed as a differential operator.  Then
		\[
		L^v_X \coloneqq 1  \otimes j^\infty(X)
		\]
		is a derivation of $dR(\sJ^\infty_\sM)$, called the \emph{vertical Lie derivative} on $dR(\sJ^\infty_\sM)$.
	\end{cor}
	
	We refer to the covariant Lie derivative $L^{\nabla^G}_X$ on $dR(\sJ^\infty_\sM)$ as the \emph{horizontal Lie derivative} and denote it by $L^h_X$.  The sum
	\[
	L_X \coloneqq L^h_X + L^v_X,
	\]
	is called the \emph{(total) Lie derivative} on $dR(\sJ^\infty_\sM)$. Note that both $L^v_X$ and $L^h_X$, and hence also $L_X$, commute with $\mathrm{d}_{\nabla^G}$.
	
	\begin{rmk}
		The \emph{Lie derivative} $L_X$ on $\sD_\sM$ along $X$ is defined by
		\[
		L_X D \coloneqq [X, D].
		\]
		$L_X$ is a coderivation of $\sD_\sM$. It extends to $dR(\sD_\sM) = \Omega_\sM \otimes_{\sO_\sM} \sD_\sM$ via the Leibniz rule:
		\[
		L_X(\alpha \otimes D) = L_X \alpha \otimes D + (-1)^{|X||\alpha|} \alpha \otimes L_X D.
		\]
		We leave it to the reader as an exercise to verify that the Lie derivative $L_X$ on $dR(\sD_\sM)$ is dual to the Lie derivative $L_X$ on $dR(\sJ^\infty_\sM)$ under the identification
		\[
		\Hom_{\sO_\sM}(\Omega_\sM \otimes_{\sO_\sM} \sD_\sM, \sO_\sM) \cong \wOmega_\sM \,\wotimes_{\sO_\sM}\, \sJ^\infty_\sM.
		\]
	\end{rmk}
	
	Let $\sE$ be a filtered graded vector bundle on $\sM$.
	
	\begin{defn}
	    The infinite jet bundle $\sJ^\infty_\sE = \sJ^\infty_\sM \,\wotimes_{\sO_\sM}\, \sE$ of $\sE$ carries a canonical flat connection $\nabla^G$, given by
		\begin{equation}\label{dje}
			\nabla^G_X(J \otimes s) \coloneqq \nabla^G_X(J) \otimes  s.
		\end{equation}
		$\nabla^G$ is called the \emph{Grothendieck connection} of $\sJ^\infty_\sE$.
	\end{defn}
	
	We need to verify that \eqref{dje} is well-defined. We have
	\[
	\nabla^G_X(J \otimes fs) = \nabla^G_X(J) \otimes  fs = (\nabla^G_X(J)j^\infty(f)) \otimes  s
	\]
	for all functions $f$. Using Lemma \ref{prol}, we have
	\[
	\nabla^G_X(J)j^\infty(f)= \nabla^G_X(Jf).
	\]
	Thus, $\nabla^G_X(J \otimes fs) = \nabla^G_X(Jf \otimes  s)$.
	
	\begin{prop}\label{resE}
		The infinite jet prolongation
		\[
		j^\infty\colon \sE \longrightarrow dR(\sJ^\infty_\sE)
		\]
		is a filtered quasi-isomorphism. That is, on the $E_0$ pages of the associated spectral sequences,
		\[
		\Gr\, j^\infty\colon \Gr\, \sE \longrightarrow \Gr\, dR(\sJ^\infty_\sE)
		\]
		is a quasi-isomorphism.
	\end{prop}
	
	\begin{prop}\label{liftE}
		Every differential operator $D_\sE\colon \sE \to \sE$ of order $k$ and filtration degree $p$ uniquely lifts to a morphism of $D$-modules 
		\[
		j^\infty(D_\sE)\colon F^\bullet \sJ^\infty_\sE \longrightarrow F^{\bullet+p-k} \sJ^\infty_\sE,
		\]
		whose associated $\N$-graded operator
		\[
		\Gr\, j^\infty(D_\sE)\colon \Sym^\bullet_{\sO_\sM}(\sT_\sM^\vee) \otimes_{\sO_\sM} \Gr^\bullet \sE \longrightarrow   \Sym^{\bullet-k}_{\sO_\sM}(\sT_\sM^\vee) \otimes_{\sO_\sM} \Gr^{\bullet+p} \sE
		\]
		is given by contraction with the principal symbol $\sigma_k^p(D_\sE) \in  \Sym^k_{\sO_\sM}(\sT_\sM) \otimes_{\sO_\sM} \Gr^p \End_\sE$ of $D_\sE$.	
		
		Moreover, the following diagram 
		\[
		\begin{tikzcd}[column sep=50 pt]
			\sJ^\infty_\sE \arrow[r, "j^\infty(D_\sE)"] & \sJ^\infty_\sE \\
			\sE \arrow[u, "j^\infty"] \arrow[r, "D_\sE"] & \sE \arrow[u, "j^\infty"]
		\end{tikzcd}
		\]
		commutes.
	\end{prop}
	
	The proofs of Propositions \ref{resE} and \ref{liftE} follow the same arguments as in the proofs of Propositions \ref{res} and \ref{lift}, and are therefore omitted.

\end{document}